\documentclass[a4]{amsart}

\input xypic 
\input xy 
\xyoption{all} 
\usepackage{amssymb} 
\usepackage{bbm}
\newtheorem{theorem}{Theorem}[section]
\newtheorem{proposition}[theorem]{Proposition}
\newtheorem{lemma}[theorem]{Lemma}

\theoremstyle{definition}
\newtheorem{definition}[theorem]{Definition}
\newtheorem{remark}[theorem]{Remark}
\newtheorem{example}[theorem]{Example}

\newcommand{\conn}{\ensuremath{\#}} 
 
\newcommand{\cpinf}{\ensuremath{\mathbb{C}P^{\infty}} }

\newcommand{\hlgy}[1]{\ensuremath{H_{*}(#1)}}

\newcommand{\cohlgy}[1]{\ensuremath{H^{*}(#1)}}

\newcounter{bean}
\newenvironment{letterlist}{\begin{list}{\rm ({\alph{bean}})}
      {\usecounter{bean}\setlength{\rightmargin}{\leftmargin}}}
      {\end{list}}

\newcommand{\namedright}[3]{\ensuremath{#1\stackrel{#2}
 {\longrightarrow}#3}}
\newcommand{\nameddright}[5]{\ensuremath{#1\stackrel{#2}
 {\longrightarrow}#3\stackrel{#4}{\longrightarrow}#5}}
\newcommand{\namedddright}[7]{\ensuremath{#1\stackrel{#2}
 {\longrightarrow}#3\stackrel{#4}{\longrightarrow}#5
  \stackrel{#6}{\longrightarrow}#7}}

\newcommand{\larrow}{\relbar\!\!\relbar\!\!\rightarrow}
\newcommand{\llarrow}{\relbar\!\!\relbar\!\!\larrow}
\newcommand{\lllarrow}{\relbar\!\!\relbar\!\!\llarrow}

\newcommand{\lnameddright}[5]{\ensuremath{#1\stackrel{#2}
 {\larrow}#3\stackrel{#4}{\larrow}#5}}

\newcommand{\llnamedright}[3]{\ensuremath{#1\stackrel{#2}
 {\llarrow}#3}}
\newcommand{\llnameddright}[5]{\ensuremath{#1\stackrel{#2}
 {\llarrow}#3\stackrel{#4}{\llarrow}#5}}

\newcommand{\lllnamedright}[3]{\ensuremath{#1\stackrel{#2}
 {\lllarrow}#3}}

\newcommand{\qqed}{\hfill\Box}

\begin{document}


\title{Top cell attachment for a Poincar\'{e} Duality complex}

\author{Stephen Theriault}
\address{Mathematical Sciences, University of Southampton, Southampton 
   SO17 1BJ, United Kingdom}
\email{S.D.Theriault@soton.ac.uk}

\subjclass[2010]{Primary 55P35, 57N65}
\keywords{Poincar\'{e} Duality complex, loop space decomposition, cell attachment}


\begin{abstract} 
Let $M$ be a simply-connected closed Poincar\'{e} Duality complex of dimension $n$. Then~$M$ 
is obtained by attaching a cell of highest dimension to its $(n-1)$-skeleton $\overline{M}$.  
Conditions are given for when the skeletal inclusion 
\(\namedright{\overline{M}}{i}{M}\) 
has the property that $\Omega i$ has a right homotopy inverse. This is an integral version of 
the rational statement that such a right homotopy inverse always exists provided the rational 
cohomology of~$M$ is not generated by a single element. New methods are developed in order 
to do the integral case. These lead to $p$-local versions and recover the full rational statement. 
Families for which the integral statement holds include moment-angle manifolds and quasi-toric manifolds. 
\end{abstract}

\maketitle

\section{Introduction} 

An important problem in homotopy theory is to determine the effect on homotopy groups 
of a cell attachment. Let $X$ be a simply-connected $CW$-complex and suppose that 
there is a homotopy cofibration 
\(\nameddright{S^{n-1}}{f}{X}{i}{Y}\) 
where $f$ attaches a cell to $X$ to obtain $Y$. It is natural to ask how $\pi_{\ast}(X)$ and $\pi_{\ast}(Y)$ 
are related. 

In rational homotopy theory the attaching map $f$ is called \emph{rationally inert} if $\pi_{\ast}(i)$ is 
a surjection. This is equivalent to saying that, rationally, the loop map 
\(\namedright{\Omega X}{\Omega i}{\Omega Y}\) 
has a right homotopy inverse. The equivalence between the rational homology of a loop space 
and the universal enveloping algebra of the rational homotopy Lie algebra of the space also allows 
for a description of rational inertness in Lie theoretic terms. Rationally inert maps have been studied 
in depth~\cite{FT,HL1,HL2}, along with their variants~\cite{B,HeL}, and play an important role in the 
growth of rational homotopy groups~\cite{FHT} and a solution to a 
higher dimensional version of Whitehead's asphericity problem~\cite{A}. 

The notion of inertness was generalized in~\cite{T} to an integral context. The attaching map $f$ 
is called \emph{inert} if the loop map 
\(\namedright{\Omega X}{\Omega i}{\Omega Y}\) 
has a right homotopy inverse. This implies that the homotopy groups of $Y$ retract off the homotopy 
groups of $X$. This is interesting because it says that the cell attachment has the special property 
of killing off homotopy groups 
but not creating any new ones. It was shown that the attaching map for the top dimensional cell 
of several families of $(n-1)$-connected $2n$-dimensional and $(n-1)$-connected $(2n+1)$-dimensional 
Poincar\'{e} duality complexes is inert. This integral version of inertness has been developed further 
in recent work~\cite{Ch1,Ch2,H,HT}. 

Halperin and Lemaire~\cite{HL1} proved the dramatic statement that the attaching map for the top 
cell of any Poincar\'{e} Duality complex $M$ is rationally inert provided the rational cohomology of~$M$ is 
not generated by a single element. In this paper we give conditions for when an analogous integral 
statement is true. It should be emphasized that in the rational case the argument depends on 
using the Sullivan algebra as a model for rational homotopy theory. No such model exists integrally, 
so new methods had to be developed. These build on work in~\cite{BT2,T} that develops 
a framework for constructing loop space decompositions in the context of homotopy fibrations 
with a section after looping. The key is to enhance this setup with a homotopy co-action that  
allows for a splitting, thereby identifying certain fibres. This should have further applications. 

To state our results, let $M$ be a simply-connected, closed $n$-dimensional Poincar\'{e} 
Duality complex. Let $\overline{M}$ be the $(n-1)$-skeleton of $M$. Then there is a homotopy cofibration 
\[\nameddright{S^{n-1}}{f}{\overline{M}}{i}{M}\] 
where $f$ attaches the top cell to $M$ and $i$ is the skeletal inclusion. 

\begin{theorem} 
   \label{introinert} 
   Let $M$ be an $(m-1)$-connected, closed Poincar\'{e} Duality complex of dimension~$n$, 
   where $2\leq m<n$. If there is a map 
   \(\namedright{M}{}{S^{m}}\) 
   having a right homotopy inverse, then the attaching map for the top cell of $M$ is inert.  
\end{theorem} 

Note that the hypothesis that $2\leq m$ implies that $M$ is at least simply-connected and $m<n$ 
implies that $M$ is not a sphere. The condition that there is a map 
\(\namedright{M}{}{S^{m}}\) 
having a right homotopy inverse is demanding but many examples of interesting Poincar\'{e} 
Duality complexes with this property are given. One family of examples is moment-angle 
manifolds, which arise in toric topology. If $K$ is a simplicial complex on $\ell$ vertices then the 
moment-angle complex associated to $K$ is constructed by gluing together $\ell$-fold products of 
factors $D^{2}$ and $S^{1}$ in a manner governed by the faces of $K$. As will be discussed in 
Examples~\ref{toricexample} and~\ref{toricexample2}, if $K$ is the triangulation of a sphere 
then the moment-angle complex is a manifold, and a missing face of least dimension, say dimension $k$, 
implies that this moment-angle manifold $M$ is $(2k-2)$-connected and there is a map 
\(\namedright{M}{}{S^{2k-1}}\) 
that has a right homotopy inverse. So Theorem~\ref{introinert} implies that the attaching map 
for the top cell of $M$ is inert. 

A $p$-local version of Theorem~\ref{introinert} also holds, which introduces more opportunities 
for producing a map 
\(\namedright{M}{}{S^{m}}\)  
having a right homotopy inverse. More is true rationally. If $M$ is rationally $(m-1)$-connected and $m$ 
is odd then such a map always exists rationally, since $m$ being odd implies that $S^{m}$ is 
rationally homotopy equivalent to the Eilenberg-Mac Lane space $K(\mathbb{Q},m)$, in which 
case producing a map 
\(\namedright{M}{}{S^{m}}\) 
is the same as identifying a generator in $H^{m}(M;\mathbb{Q})$.  

This discussion indicates that finding examples for Theorem~\ref{introinert} in which $m$ is even may 
be more delicate than when $m$ is odd, partly because the rational version of the statement runs into 
the problem that $S^{m}$ is not rationally homotopy equivalent to an Eilenberg-Mac Lane space. 
However, when $m=2$ there is an alternative statement that is based on using a map 
\(\namedright{M}{}{\cpinf}\) 
instead of a map 
\(\namedright{M}{}{S^{2}}\). 
Write $BS^{1}$ for the classifying space $\cpinf$ of $S^{1}$. 

\begin{theorem} 
   \label{introcircleinert} 
   Let $M$ be a simply-connected closed Poincar\'{e} Duality complex of dimension $n$. 
   Suppose that there is a homotopy fibration 
   \[\nameddright{N}{}{M}{h}{BS^{1}}\] 
   where $\Omega h$ has a right homotopy inverse. If the attaching map for the top cell of $N$ 
   is inert then the attaching map for the top cell of $M$ is inert. 
\end{theorem} 

Theorem~\ref{introcircleinert} may be iterated and combined with Theorem~\ref{introinert}. 
For example, let $M$ be a quasi-toric manifold. These are fundamental objects in toric 
topology associated to simple polytopes and whose construction and 
properties are thoroughly described in~\cite{BP}. One property is that, for an appropriate~$\ell$, 
there is a homotopy fibration 
\(\nameddright{N}{}{M}{}{BT^{\ell}}\) 
where $T^{\ell}$ is the torus formed by taking the product of $\ell$ copies of $S^{1}$,  
$BT^{\ell}$ is its classifying space, and $N$ is the moment-angle manifold associated to 
the simple polytope $P$. Since we have seen that the attaching map for the top cell of~$N$ is inert 
provided that $K=(\partial P)^{\ast}$ is not a simplex, by iterating Theorem~\ref{introcircleinert} 
and then applying Theorem~\ref{introinert} (for details, see Theorem~\ref{torusinert}), we obtain that the  
attaching map for the top cell of $M$ is also inert. 

Finally, the methods behind Theorem~\ref{introcircleinert} may be generalized in the rational context 
when the connectivity of $M$ is even, and combined with the earlier statement about when the 
connectivity of $M$ is odd, to reproduce Halperin and Lemaire's result that the attaching map 
for the top cell of any simply-connected, closed Poincar\'{e} Duality complex is rationally inert 
provided that its rational cohomology is not generated by a single element. It is notable that this 
is done without using Sullivan algebras. 

This paper is organized as follows. In Section~\ref{sec:BT} a theorem from~\cite{BT2} 
is recalled that sets up the context for the later sections. In particular, given the homotopy cofibration 
\(\nameddright{S^{n-1}}{}{\overline{M}}{}{M}\)  
that attaches the top cell of $M$ and a map 
\(\namedright{\overline{M}}{h}{S^{m}}\) 
such that $\Omega h$ has a right homotopy inverse and $h$ extends to a map
\(\namedright{M}{h'}{S^{m}}\), 
there is a homotopy cofibration 
\(\nameddright{S^{n-1}\rtimes\Omega S^{m}}{}{E}{}{E'}\) 
where $E$ and $E'$ are the homotopy fibres of $h$ and $h'$ respectively. Section~\ref{sec:background} 
enhances such a setup with a homotopy co-action, and in Section~\ref{sec:splitting} this is used 
to produce a splitting of the homotopy cofibration involving $E$ and $E'$. This is then applied 
in Section~\ref{sec:inert} to prove Theorem~\ref{introinert} and discuss the moment-angle manifold 
examples. Section~\ref{sec:local} gives a $p$-local version of Theorem~\ref{introinert} 
when $m$ is odd, and Section~\ref{sec:localeven} gives a $p$-local version assuming an extra 
cohomological condition when $m$ is even. We return to the integral setting in Section~\ref{sec:torus} 
where Theorem~\ref{introcircleinert} is proved and the quasi-toric manifold case is discussed. 
Section~\ref{sec:rational} generalizes Theorem~\ref{introcircleinert} in the rational setting in 
order to recover the result of Halperin and Lemaire. Finally, in Section~\ref{sec:apps} two 
applications are given, one of which identifies the homotopy fibre of the map 
\(\namedright{\overline{M}}{}{M}\) 
when the attaching map for the top cell is inert, and the other produces more examples 
of inert attaching maps via the connected sum operation.

\section{A decomposition theorem} 
\label{sec:BT} 

This section briefly states a result from~\cite{BT2}. Suppose that there is a homotopy cofibration 
\(\nameddright{A}{f}{X}{}{X'}\) 
and a homotopy fibration 
\(\nameddright{E}{}{X}{h}{Z}\). 
Suppose that $h\circ f$ is null homotopic so that $h$ extends across 
\(\namedright{X}{}{X'}\) 
to a map 
\(h'\colon\namedright{X'}{}{Z}\). 
Let $E'$ be the homotopy fibre of $h'$. This data is arranged in a diagram 
\begin{equation} 
  \label{data} 
  \diagram 
       & E\rto\dto & E'\dto \\ 
       A\rto^-{f} & X\rto\dto^{h} & X'\dto^{h'} \\ 
       & Z\rdouble & Z 
  \enddiagram 
\end{equation} 
where the two columns form a homotopy fibration diagram. 

For pointed spaces $A$ and $B$, the \emph{right half-smash} is defined as the quotient space 
\[A\rtimes B=(A\times B)/\sim\] 
where $(\ast,b)\sim(\ast,\ast)$. It is well known that if $A$ is a co-$H$-space then there is a 
homotopy equivalence $A\rtimes B\simeq A\vee (A\wedge B)$.  

\begin{theorem} 
   \label{BT} 
   Given a diagram of data~(\ref{data}). If $\Omega h$ has a right homotopy inverse 
   \(s\colon\namedright{\Omega Z}{}{\Omega X}\) 
   then there is a homotopy cofibration 
   \[\nameddright{A\rtimes\Omega Z}{\theta}{E}{}{E'}\] 
   for some map $\theta$ whose restriction to $A$ is a lift of $f$.~$\qqed$ 
 \end{theorem} 
 
 The point of Theorem~\ref{BT} is that the map $\theta$ can sometimes be used to 
 determine the homotopy type of $E'$. The right homotopy inverse for $\Omega h$ 
 implies there is a right homotopy inverse for $\Omega h'$, resulting in a homotopy 
 equivalence $\Omega X'\simeq\Omega Z\times\Omega E'$. Knowing the homotopy 
 type of $E'$ then informs on the homotopy type of $\Omega X'$.

\section{Homotopy fibrations and homotopy co-actions}  
\label{sec:background} 

The purpose of this section is to prove Proposition~\ref{co-action}. This says that if 
\(\nameddright{E}{}{X}{h}{Z}\) 
is a homotopy fibration where $\Omega h$ has a right homotopy inverse and there is a homotopy co-action  
\(\namedright{X}{}{X\vee Y}\) 
(defined later), then there is an induced homotopy co-action 
\(\namedright{E}{}{E\vee(Y\rtimes\Omega Z)}\). 
We begin with a version of Mather's Cube Lemma~\cite{Ma}. 

\begin{lemma} 
   \label{cube}
   Suppose that there is a homotopy commutative diagram of spaces and maps
   \[\spreaddiagramcolumns{-1pc}\spreaddiagramrows{-1pc} 
      \diagram
      E\rrto\drto\ddto & & F\dline\drto & \\
      & G\rrto\ddto & \dto & H\ddto \\
      A\rline\drto & \rto & B\drto & \\
      & C\rrto & & D
   \enddiagram\]
   where all four vertical maps are obtained by taking fibres over a common base space, 
   implying that the four sides are homotopy pullbacks. If the bottom face is a homotopy pushout 
   then the top face is a homotopy pushout.~$\qqed$
\end{lemma} 

Technically, Mather proved a more general result without the hypothesis that the vertical maps 
are obtained by taking fibres over a common base space, for which he needed a different definition of a 
``homotopy commutative cube". In our case the stronger hypothesis lets one use~\cite[Lemma~3.1]{PT}, 
for example, to establish Lemma~\ref{cube}.  

In general, for path-connected spaces $X$ and $Y$, let 
\(i_{1}\colon\namedright{X}{}{X\vee Y}\) 
and 
\(i_{2}\colon\namedright{Y}{}{X\vee Y}\) 
be the inclusions of the left and right wedge summands respectively and let 
\(p_{1}\colon\namedright{X\vee Y}{}{X}\) 
and 
\(p_{2}\colon\namedright{X\vee Y}{}{Y}\) 
be the pinch maps to the left and right wedge summands respectively. Let  
\(j_{1}\colon\namedright{X}{}{X\times Y}\) 
and 
\(j_{2}\colon\namedright{Y}{}{X\times Y}\) 
be the inclusions of the left and right factors respectively and let 
\(\pi_{1}\colon\namedright{X\times Y}{}{X}\) 
and 
\(\pi_{2}\colon\namedright{X\times Y}{}{Y}\) 
be the projections onto the left and right factors respectively. Let  
\(q\colon\namedright{X\times Y}{}{X\rtimes Y}\) 
be the quotient map. 

\begin{lemma}  
   \label{fibwedge} 
   Suppose that there is a homotopy fibration 
   \(\nameddright{E}{}{X}{h}{Z}\) 
   where $X$ and $Z$ are path-connected and $\Omega h$ has a right homotopy inverse. 
   Let $Y$ be another path-connected space and let $\overline{h}$ be the composite  
   \(\overline{h}\colon\nameddright{X\vee Y}{p_{1}}{X}{h}{Z}\). 
   Then there is a homotopy fibration 
   \[\nameddright{E\vee(Y\rtimes\Omega Z)}{}{X\vee Y}{\overline{h}}{Z}\] 
   and a homotopy commutative cube  
   \[\spreaddiagramrows{-1pc} 
       \diagram
          \Omega Z\rrto^-{\ast}\drto^{j_{2}}\ddto & & E\dline\drto^{i_{1}} & \\
          & Y\times\Omega Z\rrto^(0.3){i_{2}\circ q}\ddto & \dto & E\vee(Y\rtimes\Omega Z)\ddto \\
          \ast\rline\drto & \rto & X\drto & \\
          & Y\rrto & & X\vee  Y
     \enddiagram\] 
   where the bottom and top faces are homotopy pushouts, the four sides are  
   homotopy pullbacks, and the map labelled $\ast$ is null homotopic.   
\end{lemma} 

\begin{proof} 
Consider the homotopy pushout 
\[\diagram 
    \ast\rto\dto & X\dto \\ 
    Y\rto & X\vee Y. 
  \enddiagram\] 
Compose the maps in the pushout with $\overline{h}$ and take homotopy fibres. Noting 
that the restriction of~$\overline{h}$ to $X$ is $h$ while its restriction to $Y$ is null homotopic, 
we obtain homotopy fibrations 
\begin{align*} 
   \Omega Z\longrightarrow\ast & \longrightarrow Z \\ 
   E\longrightarrow X & \stackrel{h}{\longrightarrow} Z \\ 
   Y\times\Omega Z\longrightarrow Y & \longrightarrow Z \\ 
   F\longrightarrow X\vee Y & \stackrel{\overline{h}}{\longrightarrow} Z
\end{align*} 
where $F$ is defined as the homotopy fibre of $\overline{h}$. By Lemma~\ref{cube}, 
there is a homotopy commutative cube  
\begin{equation} 
   \label{XvYcube} 
   \spreaddiagramcolumns{-1pc}\spreaddiagramrows{-1pc} 
   \diagram
      \Omega Z\rrto^-{a}\drto^{b}\ddto & & E\dline\drto^{c} & \\
      & Y\times\Omega Z\rrto^(0.4){d}\ddto & \dto & F\ddto \\
      \ast\rline\drto & \rto & X\drto & \\
      & Y\rrto & & X\vee  Y
  \enddiagram 
\end{equation} 
where the bottom and top faces are homotopy pushouts and the four sides are  
homotopy pullbacks. The maps $a$, $b$, $c$ and $d$ are induced maps of fibres. 
  
First, the homotopy classes of $a$ and $b$ are identified. The rear face of the cube fits in a 
homotopy  fibration diagram  
\[\diagram 
      \Omega Z\rdouble\ddouble & \Omega Z\rto\dto^{a} & \ast\rto\dto & Z\ddouble \\ 
      \Omega Z\rto^-{\partial} & E\rto & X\rto^-{h} & Z 
  \enddiagram\]  
where  $\partial$  is the fibration connecting map. The homotopy commutativity of the  
left square implies that $a\simeq\partial$. Further, by hypothesis, $\Omega h$ has a right 
homotopy inverse, implying that $\partial$ is null homotopic. Hence $a$ is null homotopic. 
Next, the left face of the cube fits in a homotopy fibration diagram 
\[\diagram 
      \Omega Z\rdouble\ddouble & \Omega Z\rto\dto^{b} & \ast\rto\dto  & Z\ddouble \\ 
      \Omega Z\rto^{j_{2}} & Y\times\Omega Z\rto  & Y\rto & Z 
  \enddiagram\]  
where $j_{2}$ is the inclusion of the second factor. The homotopy commutativity of the 
left square implies that $b\simeq j_{2}$.  

The null homotopy for $a$ implies that the top face in~(\ref{XvYcube}) can be expanded 
to an iterated homotopy pushout  
\[\diagram
      \Omega Z\rto\dto^{i_{2}}  & \ast\rto\dto  & E\dto^{c} \\ 
      Y\times\Omega Z\rto^-{q} & Y\rtimes\Omega Z\rto  & F.    
  \enddiagram\] 
Note the bottom row is homotopic to $d$. 
The right homotopy pushout implies that \mbox{$F\simeq E\vee(Y\rtimes\Omega Z)$} and,  
under this homotopy equivalence, the maps $c$ and $d$ become $i_{1}$ and $i_{2}\circ q$ 
respectively. Substituting the identifications of $a$, $b$, $F$, $c$ and $d$ into~(\ref{XvYcube}) 
completes the proof. 
\end{proof} 

In the special case when $Z=X$ and $h$ is the identity map, Lemma~\ref{fibwedge} recovers 
the following well known result. 

\begin{lemma} 
   \label{p1fib} 
   If $X$ and $Y$ are path-connected spaces then there is a homotopy fibration 
   \[\nameddright{Y\rtimes\Omega X}{}{X\vee Y}{p_{1}}{X}.\] 
\end{lemma} 
\vspace{-1cm}~$\qqed$\bigskip 

Next, Lemma~\ref{fibwedge} is pushed further given an additional hypothesis on $X$.  

\begin{definition} 
Let $X$ and $Y$ be path-connected spaces and suppose there is a map 
\(\delta\colon\namedright{X}{}{Y}\).  
A  map  
\[\psi\colon\namedright{X}{}{X\vee Y}\] 
is called a \emph{homotopy co-action} (with respect to $\delta$) if $p_{1}\circ\psi$ is homotopic to 
the identity map on $X$ and $p_{2}\circ\psi\simeq\delta$. 
\end{definition}  

For example, a homotopy cofibration sequence 
\(\namedddright{A}{}{B}{}{C}{\delta}{\Sigma A}\) 
has a canonical homotopy co-action 
\(\namedright{C}{}{C\vee\Sigma A}\)  
with respect to $\delta$. 

Suppose as in Lemma~\ref{fibwedge} that there is a homotopy fibration 
\(\nameddright{E}{}{X}{h}{Z}\) 
and the map $\overline{h}$ is defined by the composite 
\(\overline{h}\colon\nameddright{X\vee Y}{p_{1}}{X}{h}{Z}\). 
Suppose that there are maps 
\(\namedright{X}{\delta}{Y}\) 
and 
\(\namedright{X}{\psi}{X\vee Y}\) 
where $\psi$ is a homotopy co-action with respect to $\delta$. As $\psi$ is a homotopy co-action, 
$p_{1}\circ\psi$ is homotopic to the identity map on $X$, implying that 
$\overline{h}\circ\psi=h\circ p_{1}\circ\psi\simeq h$. Observe that the naturality of the pinch 
map $p_{1}$ implies that $\overline{h}$ is also equal to the composite 
\(\nameddright{X\vee Y}{h\vee 1}{Z\vee Y}{p_{1}}{Y}\). 
Thus, writing $h$ as $p_{1}\circ(h\vee 1)\circ\psi$ we obtain an iterated homotopy pullback diagram 
\begin{equation} 
  \label{iteratedpbE} 
  \diagram 
      E\rto^-{\psi_{E}}\dto &  E\vee(Y\rtimes\Omega Z)\rto^-{t}\dto & Y\rtimes\Omega  Z\dto \\ 
      X\rto^-{\psi}\dto^{h}  & X\vee Y\rto^-{h\vee 1}\dto^{\overline{h}} &  Z\vee Y\dto^{p_{1}} \\ 
      Z\rdouble & Z\rdouble & Z 
  \enddiagram 
\end{equation} 
that defines the maps $\psi_{E}$ and $t$. Here, the homotopy fibre of $\overline{h}$ 
was identified by Lemma~\ref{fibwedge} and the homotopy fibre of $p_{1}$ was identified 
by Lemma~\ref{p1fib}. We now proceed to better identify the maps $\psi_{E}$ and $t$. 

\begin{lemma}  
   \label{tid} 
   The map $t$ in~(\ref{iteratedpbE}) is homotopic to the pinch map $p_{2}$. 
\end{lemma} 

\begin{proof} 
Consider the iterated homotopy pushout  
\[\diagram 
    \ast\rto\dto & X\rto^-{h}\dto^{i_{1}} & Z\dto^{i_{1}} \\ 
    Y\rto^-{i_{2}} & X\vee Y\rto^-{h\vee 1} & Z\vee Y.   
  \enddiagram\] 
Compose all maps in the pushout with the pinch map  
\(\namedright{Z\vee Y}{p_{1}}{Z}\)  
and take homotopy fibres. Note that $p_{1}\circ(h\vee 1)=\overline{h}$, which recovers the context 
of Lemma~\ref{fibwedge}. By Lemma~\ref{cube} applied to each homotopy pushout, and using  
Lemma~\ref{fibwedge} to identify spaces and maps in the cube for the left pushout, we obtain a 
diagram of juxtaposed homotopy commutative cubes 
\begin{equation} 
   \spreaddiagramcolumns{-1pc}\spreaddiagramrows{-1pc} 
   \diagram
      \Omega Z\rrto^-{\ast}\drto^{j_{2}}\ddto & & E\rrto\dline\drto^{i_{1}} & & \ast\dline\drto & \\
      & Y\times\Omega Z\rrto^(0.3){i_{2}\circ q}\ddto & \dto & E\vee(Y\rtimes\Omega Z)\rrto\ddto 
            & \dto & Y\rtimes\Omega Z\ddto \\
      \ast\rline\drto & \rto & X\rline^-{h}\drto & \rto & Z\drto & \\
      & Y\rrto & & X\vee  Y\rrto^-{h\vee 1} & & Z\vee Y
  \enddiagram 
\end{equation} 
where the bottom and top faces in the left and right cubes are homotopy pushouts and the four sides 
in each cube are homotopy pullbacks. The homotopy pushout in the top face of the right cube 
immediately implies that the map 
\(\namedright{E\vee(Y\rtimes\Omega Z)}{}{Y\rtimes\Omega Z}\)  
is homotopic to $p_{2}$. Note that the front face of the right cube is exactly the homotopy 
pullback in the upper right square of~(\ref{iteratedpbE}). Thus 
the map~$t$ in~(\ref{iteratedpbE}) is homotopic to $p_{2}$, as asserted. 
\end{proof} 

The map $\psi_{E}$ covering the homotopy co-action $\psi$ has another property. 
By definition, $\overline{h}=h\circ p_{1}$, so we obtain an iterated homotopy pullback diagram 
\begin{equation} 
  \label{iteratedpbE2} 
  \diagram 
      E\rto^-{\psi_{E}}\dto &  E\vee(Y\rtimes\Omega Z)\rto^-{r}\dto & E\dto \\ 
      X\rto^-{\psi}\dto^{h}  & X\vee Y\rto^-{p_{1}}\dto^{\overline{h}} &  X\dto^{h} \\ 
      Z\rdouble & Z\rdouble & Z 
  \enddiagram 
\end{equation} 
that defines the map $r$. 

\begin{lemma} 
   \label{rid} 
   The map $r$ in~(\ref{iteratedpbE2}) is homotopic to the pinch map $p_{1}$.  
\end{lemma} 

\begin{proof}  
Consider the iterated homotopy pushout  
\[\diagram 
    \ast\rto\dto & Y\rto\dto^{i_{2}} & \ast\dto^{i_{1}} \\ 
    X\rto^-{i_{1}} & X\vee Y\rto^-{p_{1}} & X.   
  \enddiagram\] 
Compose the maps in the pushout with 
\(\namedright{X}{h}{Z}\) 
and take homotopy fibres. Note that $p_{1}\circ h=\overline{h}$, which recovers the context of 
Lemma~\ref{fibwedge}. By Lemma~\ref{cube} applied to each homotopy pushout, and using 
Lemma~\ref{fibwedge} to identify spaces and maps in the cube for the left pushout, we obtain 
a diagram of juxtaposed homotopy commutative cubes 
\begin{equation} 
   \label{XvYcube3} 
   \spreaddiagramcolumns{-1pc}\spreaddiagramrows{-1pc} 
   \diagram
      \Omega Z\rrto^-{j_{2}}\drto^{\ast}\ddto & & Y\times\Omega Z\rrto^-{a}\dline\drto^{i_{2}\circ q} 
            & & \Omega Z\dline\drto & \\
      & E\rrto^(0.3){i_{1}}\ddto & \dto & E\vee(Y\rtimes\Omega Z)\rrto^(0.5){r}\ddto & \dto & E\ddto \\
      \ast\rline\drto & \rto & Y\rline\drto & \rto & \ast\drto & \\
      & X\rrto & & X\vee  Y\rrto^-{p_{1}} & & X
  \enddiagram 
\end{equation} 
where the bottom and top faces in the left and right cubes are homotopy pushouts, the four sides 
in each cube are homotopy pullbacks, and the map $a$ is an induced map of homotopy fibres. 
Notice the front face in the right cube is the homotopy pullback in the upper right square 
of~(\ref{iteratedpbE2}), thereby identifying $r$.  

We begin by identifying the map $a$. Since the composite 
\(\namedddright{Y}{i_{2}}{X\vee Y}{p_{1}}{X}{h}{Z}\) 
is null homotopic, the rear face of the right cube fits in a homotopy fibration diagram  
\[\diagram 
     \Omega Z\rto^{j_{2}}\ddouble  & Y\times\Omega Z\rto^-{\pi_{1}}\dto^{a} & Y\rto^{\ast}\dto & Z\ddouble \\ 
     \Omega Z\rdouble & \Omega Z\rto & \ast\rto & Z.  
  \enddiagram\]  
The middle square is a homotopy pullback, implying that $a\simeq\pi_{2}$. With $a$ identified, 
the fact that $i_{2}\circ q$ is a composite implies that the right top face in~(\ref{XvYcube3}) expands 
to an iterated homotopy pushout  
\begin{equation} 
  \label{psip2finish} 
  \diagram 
      Y\times\Omega Z\rto^-{\pi_{2}}\dto^{q} & \Omega Z\dto \\ 
      Y\rtimes\Omega Z\rto\dto^{i_{2}} & Q\dto \\ 
      E\vee(Y\rtimes\Omega Z)\rto^-{r} & E. 
  \enddiagram 
\end{equation}  
that defines the space $Q$. Consider the homotopy pushout in the top square of~(\ref{psip2finish}). 
Since $q$ is obtained by taking the homotopy cofibre of the inclusion 
\(\namedright{\Omega Z}{j_{2}}{Y\times\Omega Z}\), 
the pushout implies that there is a homotopy cofibration 
\(\llnameddright{\Omega Z}{\pi_{2}\circ j_{2}}{\Omega Z}{}{Q}\). 
As $\pi_{2}\circ q$ is homotopic to the identity map on $\Omega Z$, the space $Q$ is contractible. 
The bottom homotopy pushout in~(\ref{psip2finish}) then implies that there is a 
homotopy cofibration 
\(\nameddright{Y\rtimes\Omega Z}{i_{2}}{E\vee(Y\rtimes\Omega Z)}{r}{E}\), 
implying that $r$ is homotopic to $p_{1}$. 
\end{proof}  

Define the map $\delta_{E}$ by the homotopy pullback diagram 
\begin{equation} 
  \label{deltaEpb} 
  \diagram 
      E\rto^-{\delta_{E}}\dto  & Y\rtimes\Omega Z\dto \\ 
      X\rto^-{(h\vee 1)\circ\psi}\dto^{h} & Z\vee Y\dto^{p_{1}} \\ 
      Z\rdouble & Z. 
  \enddiagram 
\end{equation} 

\begin{proposition} 
   \label{co-action} 
   Given the hypotheses of Lemma~\ref{fibwedge}, suppose as well that there is a map 
   \(\namedright{X}{\delta}{Y}\) 
   and a homotopy co-action  
   \(\namedright{X}{\psi}{X\vee Y}\)  
   with respect to $\delta$. Then there is a homotopy fibration diagram  
   \[\diagram 
         E\rto^-{\psi_{E}}\dto & E\vee(Y\rtimes\Omega Z)\dto \\ 
         X\rto^-{\psi}\dto^{h} & X\vee Y\dto^{\overline{h}} \\ 
         Z\rdouble & Z  
     \enddiagram\] 
   where  $\psi_{E}$ is a homotopy co-action with respect to $\delta_{E}$.  
\end{proposition}  

\begin{proof} 
First consider the iterated homotopy pullback in~(\ref{iteratedpbE2}). Since $\psi$  
is a homotopy co-action, $p_{1}\circ\psi$ is homotopic to the identity map on $X$. 
Therefore $r\circ\psi_{E}$ is homotopic to the identity map on $E$. But by 
Lemma~\ref{rid}, $r\simeq p_{1}$, so $p_{1}\circ\psi_{E}$ is homotopic to the 
identity map on $E$. Next, observe that the homotopy pullback defining $\delta_{E}$ 
is the iterated homotopy pullback in~(\ref{iteratedpbE}). Thus $t\circ\psi_{E}\simeq\delta_{E}$. 
Lemma~\ref{tid} then implies that $p_{2}\circ\psi_{E}\simeq\delta_{E}$. Hence $\psi_{E}$ 
is a homotopy co-action with respect to~$\delta_{E}$. 
\end{proof}

\section{A splitting criterion} 
\label{sec:splitting} 

Given data as in~(\ref{data}) having the property that $\Omega h$ has a right homotopy 
inverse, Theorem~\ref{BT} states that there is a homotopy cofibration 
\[\nameddright{A\rtimes\Omega Z}{\theta}{E}{}{E'}.\]
Ideally, properties of $\theta$ can be determined that inform on the homotopy type of $E$ or $E'$. 
In this section a criterion is proved that implies $\theta$ has a left homotopy inverse under certain 
hypotheses and there is a homotopy equivalence $E\simeq (A\rtimes\Omega Z)\vee E'$. 

Suppose that there is a map 
\(\delta\colon\namedright{X}{}{\Sigma Y}\)  
and a homotopy co-action 
\[\psi\colon\namedright{X}{}{X\vee\Sigma Y}\] 
with respect to $\delta$. Suppose as well that there is a diagram of data 
\begin{equation} 
  \label{criteriondata} 
  \diagram 
       & E\rto\dto & E'\dto  \\ 
       \Sigma Z\wedge Y\rto^-{f} & X\rto\dto^{h} & X'\dto^{h'} \\ 
       & \Sigma Z\rdouble & \Sigma Z 
  \enddiagram 
\end{equation} 
where the middle row is a homotopy cofibration and the two columns form 
a homotopy fibration diagram. Let $\gamma$ be the composite 
\begin{equation} 
  \label{gammadef} 
  \gamma\colon\nameddright{X}{\psi}{X\vee\Sigma Y}{h\vee 1}{\Sigma Z\vee\Sigma Y}.  
\end{equation} 
Since $\psi$ is a co-action, the composite $p_{1}\circ\gamma$ is homotopic to $h$.  
Define the space $D$ by the homotopy cofibration diagram 
\begin{equation} 
  \label{Ddef} 
  \diagram 
        \Sigma Z\wedge Y\rto^-{f}\ddouble & X\rto\dto^{\gamma} & X'\dto \\ 
        \Sigma Z\wedge Y\rto^-{\gamma\circ f} & \Sigma Z\vee\Sigma Y\rto &  D. 
  \enddiagram 
\end{equation}  

The prototype to think of is  when $\gamma\circ f$ is homotopic to the Whitehead 
product of the inclusions of $\Sigma Z$ and $\Sigma Y$ into $\Sigma Z\vee\Sigma Y$,  
in which case $D\simeq\Sigma Z\times\Sigma Y$. However, we wish to allow for more  
flexibility in terms of the homotopy class of $\gamma\circ f$. To get this, observe that as 
$p_{1}\circ\gamma\simeq h$ and $h\circ f$ is null homotopic, there is a null  homotopy 
for $p_{1}\circ\gamma\circ f$, implying that the pinch map  
\(\namedright{\Sigma Z\vee\Sigma Y}{p_{1}}{\Sigma Z}\) 
extends to a map 
\(\namedright{D}{\mathfrak{h}}{\Sigma Z}\). 
Let $g$ be the composite 
\[g\colon\nameddright{\Sigma Y}{i_{2}}{\Sigma Z\vee\Sigma Y}{}{D}.\] 
Then in place of $\gamma\circ f$ being a Whitehead product, giving $D\simeq\Sigma Z\times\Sigma Y$ 
and a homotopy fibration 
\(\nameddright{\Sigma Y}{g}{D}{\mathfrak{h}}{\Sigma Z}\), 
we only assume the existence of the homotopy fibration. 

\begin{theorem}  
   \label{splittingprinciple} 
   Suppose that there is a map 
   \(\delta\colon\namedright{X}{}{\Sigma Y}\)  
   and a homotopy co-action 
   \(\psi\colon\namedright{X}{}{X\vee\Sigma Y}\) 
   with respect to $\delta$, and data as in~(\ref{criteriondata}). If: 
   \begin{letterlist} 
      \item $\Omega h$ has a right homotopy inverse and 
      \item there is a homotopy fibration 
               \(\nameddright{\Sigma Y}{g}{D}{\mathfrak{h}}{\Sigma Z}\), 
   \end{letterlist}    
   then the homotopy cofibration 
   \(\nameddright{(\Sigma Z\wedge Y)\rtimes\Omega\Sigma Z}{\theta}{E}{}{E'}\)  
   obtained by applying Theorem~\ref{BT} to~(\ref{criteriondata}) splits: the map $\theta$ has 
   a left homotopy inverse and there is a homotopy equivalence 
   \[E\simeq((\Sigma Z\wedge Y)\rtimes\Omega\Sigma Z)\vee  E'.\] 
\end{theorem} 

\begin{proof} 
Start with the homotopy pushout in the right square of~(\ref{Ddef}). Compose the maps in  
the pushout with 
\(\namedright{D}{\mathfrak{h}}{\Sigma Z}\) 
and take homotopy fibres. Note that, by definition, $\mathfrak{h}$ extends the pinch map 
\(\namedright{\Sigma Z\vee\Sigma Y}{p_{1}}{\Sigma Z}\), 
and as mentioned just before~(\ref{Ddef}), $p_{1}\circ\gamma\simeq h$. The map 
\(\namedright{X'}{h'}{\Sigma Z}\) 
was defined as an extension of $h$ and its homotopy fibre was named $E'$; this involved 
a choice of extension, and one choice is the composite  
\(\nameddright{X'}{}{D}{\mathfrak{h}}{\Sigma Z}\). 
Taking this choice of $h'$ we obtain homotopy fibrations 
\begin{align*} 
   E\longrightarrow X & \stackrel{h}{\longrightarrow}\Sigma Z \\ 
   E'\longrightarrow X'  & \stackrel{h'}{\longrightarrow}\Sigma Z \\  
   \Sigma Y\rtimes\Omega\Sigma Z\longrightarrow\Sigma Z\vee\Sigma Y  
        & \stackrel{p_{1}}{\longrightarrow}\Sigma Z \\ 
   \Sigma B\longrightarrow D  & \stackrel{\mathfrak{h}}{\longrightarrow}\Sigma Z 
\end{align*} 
where the last homotopy fibration holds by hypothesis~(b). By Lemma~\ref{cube}, there 
is a homotopy commutative cube 
\begin{equation} 
   \label{splittingcube} 
   \spreaddiagramcolumns{-1pc}\spreaddiagramrows{-1pc} 
   \diagram
      E\rrto\drto^{\alpha}\ddto & & E'\dline\drto & \\
      & \Sigma Y\rtimes\Omega\Sigma Z\rrto\ddto & \dto & \Sigma Y\ddto \\
      X\rline\drto^{\gamma} & \rto & X'\drto & \\
      & \Sigma Z\vee\Sigma Y\rrto & & D
  \enddiagram 
\end{equation} 
where the bottom and top faces are homotopy pushouts, the top face is a homotopy pullback, 
and the homotopy pullback in the left face defines $\alpha$. 
\medskip  

\noindent 
\textit{Step 1: Identifying $\alpha$}. Consider the iterated homotopy fibration diagram 
\begin{equation}  
  \label{alphaid} 
  \diagram 
     E\rto^-{\psi_{E}}\dto & E\vee(\Sigma Y\rtimes\Omega\Sigma Z)\rto^-{p_{2}}\dto 
         & \Sigma Y\rtimes\Omega\Sigma Z\dto \\ 
     X\rto^-{\psi}\dto & X\vee\Sigma Y\rto^-{h\vee 1}\dto^{\overline{h}} & \Sigma Z\vee\Sigma Y\dto^{p_{1}} \\ 
     \Sigma Z\rdouble & \Sigma Z\rdouble & \Sigma Z 
  \enddiagram 
\end{equation}  
where $\overline{h}=p_{1}\circ(h\vee 1)$ and the induced map of fibres in the right rectangle 
was identified as $p_{2}$ in Lemma~\ref{rid}. By Proposition~\ref{co-action}, $\psi_{E}$ is a 
homotopy co-action with respect to $\delta_{E}$, where $\delta_{E}=p_{2}\circ\psi_{E}$. 
By definition, $\gamma=(h\vee 1)\circ\psi$ is the composite along the middle row 
of~(\ref{alphaid}), so the iterated homotopy pullbacks in the upper squares of~(\ref{alphaid}) 
show that pulling $\gamma$ back with 
\(\namedright{\Sigma Y\rtimes\Omega Z}{}{\Sigma Z\vee\Sigma Y}\) 
determines the homotopy class of 
\(\namedright{E}{}{\Sigma Y\rtimes\Omega\Sigma Z}\). 
On the one hand, this is $\delta_{E}$, and on the other hand, this is the map $\alpha$ 
in~(\ref{splittingcube}). Therefore $\alpha\simeq\delta_{E}$. 
\medskip 

\noindent 
\textit{Step 2: An expanded homotopy pushout}.
The rear face in the cube~(\ref{splittingcube}) fits in the diagram of data~(\ref{criteriondata}). Thus by 
Theorem~\ref{BT} there is a homotopy cofibration 
\[\nameddright{(\Sigma Z\wedge Y)\rtimes\Omega\Sigma Z}{\theta}{E}{}{E'}\]  
for some map $\theta$. The homotopy pushout in the top face of the cube~(\ref{splittingcube}) can therefore  
be expanded to a homotopy cofibration diagram 
\begin{equation} 
  \label{extendedtheta} 
  \diagram 
      (\Sigma Z\wedge Y)\rtimes\Omega\Sigma Z\rto^-{\theta}\ddouble 
            & E\rto\dto^{\alpha} & E'\dto \\ 
      (\Sigma Z\wedge Y)\rtimes\Omega\Sigma Z\rto^-{\alpha\circ\theta} 
            & \Sigma Y\rtimes\Omega\Sigma Z\rto & \Sigma Y. 
  \enddiagram 
\end{equation}  
\medskip  

\noindent 
\textit{Step 3: a left homotopy inverse for $\theta$}: 
The homotopy pullback in the front face of the cube implies that there is a pullback map $t$  
making the following diagram homotopy commute
\begin{equation} 
  \label{SigmaBsplitting} 
  \xymatrix{ 
     \Sigma  Y\ar@/^/[drr]^{=}\ar@/_/[ddr]_{i_{2}}\ar@{.>}[dr]^{t} & & \\ 
     & \Sigma Y\rtimes\Omega\Sigma Z\ar[r]\ar[d] & \Sigma Y\ar[d] \\ 
     & \Sigma Z\vee\Sigma Y\ar[r] & D. }
\end{equation} 
Consequently, $t$ is a section for the homotopy cofibration  
\(\nameddright{(\Sigma Z\wedge Y)\rtimes\Omega\Sigma Z}{\alpha\circ\theta} 
      {\Sigma Y\rtimes\Omega\Sigma Z}{}{\Sigma Y}\), 
implying that the map 
\[\lllnamedright{\Sigma Y\vee((\Sigma Y\wedge Z)\rtimes\Omega\Sigma Z)}{t\perp(\alpha\circ\theta)} 
      {\Sigma Y\rtimes\Omega\Sigma Z}\] 
induces an isomorphism in homology and so is a homotopy equivalence by Whitehead's theorem.      
This implies that $\alpha\circ\theta$ has a left homotopy inverse 
\(r\colon\namedright{\Sigma Y\rtimes\Omega\Sigma Z}{}{(\Sigma Z\wedge Y)\rtimes\Omega\Sigma Z}\). 
The homotopy commutativity of the left square in~(\ref{extendedtheta}) then implies that 
$r\circ\alpha$ is a left homotopy inverse for $\theta$. 
\medskip 

\noindent 
\textit{Step 4: a wedge decomposition for $E$}.  
Finally, we wish to show that the left homotopy inverse for $\theta$ implies a splitting for $E$. 
In general, a homotopy cofibration 
\(\nameddright{A}{a}{B}{}{C}\)  
with a left homotopy inverse for $a$ does not result in a homotopy equivalence $B\simeq A\vee C$ 
unless there is a way to ``add" the maps 
\(\namedright{B}{}{A}\)  
and 
\(\namedright{B}{}{C}\) 
to produce a map 
\(\namedright{B}{}{A\vee C}\).  
This works if $B$ is a co-$H$-space but in our case, $E$ may not be a co-$H$-space. However, 
by Lemma~\ref{co-action} there is a homotopy co-action  
\(\namedright{E}{\psi_{E}}{E\vee(\Sigma Y\rtimes\Omega\Sigma Z)}\)  
and this will be used as the means to ``add".

Consider again the homotopy cofibration 
\(\nameddright{(\Sigma Z\wedge Y)\rtimes\Omega\Sigma Z}{\theta}{E}{g}{E'}\), 
where $g$ is simply a name for the map from $E$ to $E'$. By Step~$1$, $\alpha\simeq\delta_{E}$. 
Substituting this into Step~$3$ implies that $r\circ\delta_{E}$ 
is a left homotopy inverse for $\theta$. Now consider the composite 
\[e\colon\lnameddright{E}{\psi_{E}}{E\vee(\Sigma Y\rtimes\Omega\Sigma Z)}{g\vee r} 
       {E'\vee((\Sigma Z\wedge Y)\rtimes\Omega\Sigma Z)}\] 
where $r$ is the map from Step~$3$.  
Since $\psi_{E}$ is a co-action, $p_{1}\circ\psi_{E}$ is homotopic to the identity map on $E$ 
and $p_{2}\circ\psi_{E}$ is homotopic to~$\delta_{E}$. The naturality of $p_{1}$ and $p_{2}$  
then imply that $p_{1}\circ e\simeq g$ and $p_{2}\circ e\simeq r\circ\delta_{E}$. As $r\circ\delta_{E}$ 
is a left homotopy inverse for $\theta$, the map $e$ induces an isomorphism in homology and 
so is a homotopy equivalence by Whitehead's theorem. 
\end{proof}

\section{Inert attaching maps I} 
\label{sec:inert} 

Theorem~\ref{splittingprinciple} will be applied in the context of Poincar\'{e} Duality complexes. 
This will take a bit of setting up. Let $M$ be a simply-connected closed Poincar\'{e} 
Duality complex of dimension~$n$, where~$n\geq 4$. 
Let~$\overline{M}$ be the $(n-1)$-skeleton of $M$. Then there is a homotopy cofibration 
\[\nameddright{S^{n-1}}{f}{\overline{M}}{i}{M}\] 
where $f$ attaches the top cell to $M$ and $i$ is the skeletal inclusion. 

Suppose that $M$ is $(m-1)$-connected, for $m\geq 2$. Note that Poincar\'{e} Duality implies 
that $n\geq 2m$, so $n-m\geq m$. Suppose that there is a map 
\[h'\colon\namedright{M}{}{S^{m}}\] 
with a right homotopy inverse. If $\iota\in H^{m}(S^{m})\cong\mathbb{Z}$ represents a generator, 
let $x=(h')^{\ast}(\iota)\in H^{m}(M)$. Since $h'$ has a right homotopy inverse, $x$ generates 
a $\mathbb{Z}$-summand in $H^{m}(M)$. By Poincar\'{e} Duality, there is a class $y\in H^{n-m}(M)$ 
generating a $\mathbb{Z}$-summand and such that $x\cup y=z$, where~$z$ is a generator of $H^{n}(M)$ 
(for example, see~\cite[Corollary 3.39]{H}). It may be worth noting that as $h'$ has a right homotopy 
inverse, $x^{2}=0$, so $y\neq x$. 

Let $h$ be the composite 
\[h\colon\nameddright{\overline{M}}{i}{M}{h'}{S^{m}}.\]  
As $s$ factors through the $(n-1)$-skeleton of $M$, the map $h$ also has a right homotopy inverse. 
As in~(\ref{data}), we obtain a diagram of data 
\begin{equation} 
  \label{Mdata} 
  \diagram 
       & E\rto\dto & E'\dto \\ 
       S^{n-1}\rto^-{f} & \overline{M}\rto^-{i}\dto^{h} & M\dto^{h'} \\ 
       & S^{m}\rdouble & S^{m}  
  \enddiagram 
\end{equation} 
that defines the spaces $E$ and $E'$. (As a reminder, the middle row in this diagram is a homotopy 
cofibration and the two columns form a homotopy fibration diagram.) Since $h$ has a right homotopy 
inverse, by Theorem~\ref{BT} there is a homotopy cofibration 
\begin{equation} 
  \label{Mdatacofib}  
  \nameddright{S^{n-1}\rtimes\Omega S^{m}}{\theta}{E}{}{E'}  
\end{equation} 
for some map $\theta$. In Proposition~\ref{Esplitting} we will find conditions that guarantee that this 
homotopy cofibration splits. 

For $k\geq 1$, let $M_{k}$ be the $k$-skeleton of $M$. Since $M$ is $(m-1)$-connected,  
we have $H_{i}(M)\cong 0$ for $0<i<m$. Poincar\'{e} Duality then implies that $H^{i}(M)\cong 0$ 
for $m-n<i<n$, and the universal coefficient theorem therefore implies that $H_{i}(M)\cong 0$ 
for $m-n<i<n$. Thus $\overline{M}\simeq M_{n-m}$. Further, we are assuming that $H_{m}(M)$ 
has a $\mathbb{Z}$-summand, so Poincar\'{e} Duality implies that $H^{n-m}(M)$ has a 
$\mathbb{Z}$-summand, and the universal coefficient theorem then implies that $H_{n-m}(M)$ 
has a $\mathbb{Z}$-summand. Thus $\overline{M}$ has at least one cell in dimension $n-m$. 

There is a homotopy cofibration that attaches the $(n-m)$-cells to $\overline{M}$ of the form 
\[\nameddright{\bigvee_{i=1}^{d}  S^{n-m-1}}{}{M_{n-m-1}}{}{\overline{M}}.\] 
As $\overline{M}$ has at least one $(n-m)$-cell, we have $d\geq 1$. This homotopy cofibration 
has a connecting map 
\[\namedright{\overline{M}}{\delta'}{\bigvee_{i=1}^{d} S^{n-m}}\] 
and a homotopy co-action 
\[\psi'\colon\namedright{\overline{M}}{}{\overline{M}\vee(\bigvee_{i=1}^{d}  S^{n-m})}\] 
with respect to $\delta'$. Identify $y\in H^{n-m}(M)$ with the corresponding class in $H^{n-m}(\overline{M})$. 
Since $\delta'$ collapses out the $M_{n-m-1}$-skeleton, $y$ is in the image of $(\delta')^{\ast}$. 
Changing $\bigvee_{i=1}^{d} S^{n-m-1}$ by a self-equivalence if necessary, we may 
assume that the restriction of~$(\delta')^{\ast}$ to the $i=1$ summand has image $y$. 
Let $\delta$ be the composite 
\[\delta\colon\nameddright{\overline{M}}{\delta'}{\bigvee_{i=1}^{d} S^{n-m}}{p_{1}}{S^{n-m}}\] 
where $p_{1}$ is the pinch map to the first wedge summand. Then $\delta^{\ast}$ has 
image $y$. Let $\psi$ be the composite 
\[\psi\colon\lnameddright{\overline{M}}{\psi'}{\overline{M}\vee(\bigvee_{i=1}^{d} S^{n-m})}{1\vee p_{1}}  
      {\overline{M}\vee S^{n-m}}.\]  
It is straightforward to check that $\psi'$ being a homotopy co-action with respect to $\delta'$ 
implies that $\psi$ is a homotopy co-action with respect to $\delta$. 

\begin{proposition} 
   \label{Esplitting} 
   Let $M$ be an $(m-1)$-connected closed Poincar\'{e} Duality complex of 
   dimension~$n$, where~$2\leq m<n$. If there is a map 
   \(\namedright{M}{h'}{S^{m}}\) 
   having a right homotopy inverse, then the homotopy cofibration~(\ref{Mdatacofib}) splits 
   and there is a homotopy equivalence $E\simeq(S^{n-1}\rtimes\Omega S^{m})\vee E'$. 
\end{proposition} 

\begin{proof} 
In the context of Theorem~\ref{splittingprinciple}, we start with the map 
\(\namedright{\overline{M}}{\delta}{S^{n-m}}\) 
and the homotopy co-action 
\(\namedright{\overline{M}}{\psi}{\overline{M}\vee S^{n-m}}\).  
The diagram of data~(\ref{Mdata}) corresponds to~(\ref{criteriondata}) with $X=\overline{M}$, 
$X'=M$, $\Sigma Z=S^{m}$ and $\Sigma Y=S^{n-m}$. The map $\gamma$ in~(\ref{gammadef}) is the composite 
\[\gamma\colon\nameddright{\overline{M}}{\psi}{\overline{M}\vee S^{n-m}}{h\vee 1}{S^{m}\vee S^{n-m}}\] 
and there is a homotopy cofibration diagram: 
\begin{equation} 
  \label{Dpo} 
  \diagram 
     S^{n-1}\rto^-{f}\ddouble & \overline{M}\rto^-{i}\dto^{\gamma} & M\dto^{\gamma'} \\ 
     S^{n-1}\rto^-{\gamma\circ f} & S^{m}\vee S^{n-m}\rto & D 
  \enddiagram 
\end{equation} 
that defines the space $D$ and the map $\gamma'$. The aim is to show that conditions~(a) 
and~(b) of Theorem~\ref{splittingprinciple} hold in order to obtain a splitting of the homotopy  
cofibration~(\ref{Mdatacofib}). 

First, we describe $H^{\ast}(D)$. The homotopy cofibration in the bottom row of~(\ref{Dpo}) 
implies that $H^{\ast}(D)$ is isomorphic as a $\mathbb{Z}$-module to $\mathbb{Z}\{a,b,c\}$ 
where $\vert a\vert=m$, $\vert b\vert=n-m$ and $\vert c\vert=n$. Since the map 
\(\namedright{S^{m}\vee  S^{n-m}}{}{D}\) 
is the inclusion of the $(n-1)$-skeleton, it induces an isomorphism on $H^{m}$ and $H^{n-m}$. 
So we may write $H^{\ast}(S^{m}\vee  S^{n-m})$ as $\mathbb{Z}\{a,b\}$. By definition, 
$\gamma=(h\vee 1)\circ\psi$. As $h$ is the restriction of $h'$ to $\overline{M}$ and $(h')^{\ast}$ 
has image $x$, we obtain $h^{\ast}(a)=x$ implying that $\gamma^{\ast}(a)=x$. The 
co-action $\psi$ is with respect to $\delta$ and the image of $\delta^{\ast}$ is $y$, 
therefore $\gamma^{\ast}(b)=y$. The homotopy commutativity of the right square 
in~(\ref{Dpo}) then implies that $(\gamma')^{\ast}(a)=x$ and $(\gamma')^{\ast}(b)=y$. 
The homotopy cofibration diagram~(\ref{Dpo}) induces a map of long exact sequences in 
cohomology which, in degree~$n$, shows that $(\gamma')^{\ast}$ is an isomorphism 
on $H^{n}$, implying that $(\gamma')^{\ast}(c)=z$. Finally, since $x\cup y=z$, where $z$ generates 
$H^{n}(M)$, and $(\gamma')^{\ast}$ is an algebra map, we must have $a\cup b=c$ in $H^{\ast}(D)$. 
Thus there is an algebra isomorphism $H^{\ast}(D)\cong H^{\ast}(S^{m}\times S^{n-m})$.  

Next, we show that $D$ is a sphere bundle over a sphere. Since $p_{1}\circ\gamma\simeq h$ 
and $h$ extends across~$i$ to 
\(\namedright{M}{h'}{S^{m}}\), 
there is a pushout map 
\(\mathfrak{h}\colon\namedright{D}{}{S^{m}}\)  
making the following diagram homotopy commute  
\[\xymatrix{ 
     \overline{M}\ar[r]^{i}\ar[d]^{\gamma} & M\ar@/^/[ddr]^{h'}\ar[d] & \\ 
     S^{m}\vee S^{n-m}\ar[r]\ar@/_/[drr]_{p_{1}} 
         & D\ar@{.>}[dr]^(0.4){\mathfrak{h}} & \\ 
     & & S^{m}. }
\] 
As $\mathfrak{h}$ extends $h'$ and $(h')^{\ast}$ has image $x$, the image of 
$\mathfrak{h}^{\ast}$ is $a$. Let $F$ be the homotopy fibre of $\mathfrak{h}$.  
As $H^{\ast}(D)\cong H^{\ast}(S^{m}\times S^{n-m})\cong\Lambda(a,b)$ and 
$\mathfrak{h}^{\ast}$ induces the inclusion of $\Lambda(a)$, the Serre spectral sequence 
implies that $H^{\ast}(F)\cong\Lambda(b)\cong H^{\ast}(S^{n-m})$. Dualizing, 
$H_{\ast}(F)\cong H_{\ast}(S^{n-m})$, so the inclusion 
\(\namedright{S^{n-m}}{}{F}\) 
of the bottom cell induces an isomorphism in homology and is therefore a homotopy 
equivalence by Whitehead's Theorem. Hence there is a homotopy fibration 
\(\nameddright{S^{n-m}}{}{D}{\mathfrak{h}}{S^{m}}\). 

Finally, by hypothesis, 
\(\namedright{M}{h'}{S^{m}}\) 
has a right homotopy inverse 
\(\namedright{S^{m}}{}{M}\). 
For skeletal reasons, the latter map factors as 
\(\nameddright{S^{m}}{}{\overline{M}}{i}{M}\), 
implying that $h=h'\circ i$ has a right homotopy inverse. 
Thus both hypotheses~(a) and~(b) of Theorem~\ref{splittingprinciple} 
are satisfied. The theorem states that the homotopy cofibration~(\ref{Mdatacofib}) splits 
and there is a homotopy equivalence $E\simeq(S^{n-1}\rtimes\Omega S^{m})\vee E'$. 
\end{proof} 

Recall that, for the homotopy cofibration 
\(\nameddright{S^{n-1}}{f}{\overline{M}}{i}{M}\), 
the attaching map $f$ for the top cell is \emph{inert} if $\Omega i$ has a right homotopy inverse. 

\begin{theorem}[Theorem~\ref{introinert} in the Introduction]  
   \label{inert} 
   Let $M$ be an $(m-1)$-connected, closed Poincar\'{e} Duality complex of dimension~$n$, 
   where $2\leq m<n$. If there is a map 
   \(\namedright{M}{}{S^{m}}\) 
   having a right homotopy inverse, then the attaching map for the top cell of $M$ is inert.  
\end{theorem} 
   
\begin{proof} 
Note that the hypothesis that $2\leq m$ implies that $M$ is at least simply-connected and $m<n$ 
implies that $M$ is not a sphere. The inertness property is equivalent to showing that the map  
\(\namedright{\Omega\overline{M}}{\Omega i}{\Omega M}\) 
has a right homotopy inverse. Consider the homotopy fibration diagram
\[\diagram 
     E\rto\dto^{e} & E'\dto^{e'} \\ 
     \overline{M}\rto^-{i}\dto^{h} & M\dto^{h'} \\ 
     S^{m}\rdouble & S^{m}  
  \enddiagram\] 
where $e$ and $e'$ are names for the maps from the homotopy fibres of $h$ and $h'$ respectively. 
Since $h'$ has a right homotopy inverse 
\(s'\colon\namedright{S^{m}}{}{M}\), 
after looping in order to multiply, the composite 
\[\epsilon\colon\llnameddright{\Omega S^{m}\times\Omega E}{\Omega s'\times\Omega e'}  
      {\Omega M\times\Omega M}{\mu}{\Omega M}\] 
is a homotopy equivalence, where $\mu$ is the canonical loop multiplication.  

For degree reasons, the left homotopy inverse 
\(s'\colon\namedright{S^{m}}{}{M}\) 
for $h'$ lifts through $i$ to a map 
\(s\colon\namedright{S^{m}}{}{\overline{M}}\).  
By Proposition~\ref{Esplitting}, the map 
\(\namedright{E}{}{E'}\) 
has a right homotopy inverse 
\(t\colon\namedright{E'}{}{E}\). 
Thus $e\circ t$ is a lift of $e'$ through $i$. The lifts of $s'$ and $e'$ through $i$, together 
with $\Omega i$ being an $H$-map, imply that there is a homotopy commutative diagram 
\[\diagram
      \Omega S^{m}\times\Omega E'\rrto^-{\Omega s\times\Omega(e\circ t)}\drrto_{\Omega s'\times\Omega e'}  
          & & \Omega\overline{M}\times\Omega\overline{M}\rto^-{\mu}\dto^{\Omega i\times\Omega i} 
          & \Omega\overline{M}\dto^{\Omega i} \\ 
      & & \Omega M\times\Omega M\rto^-{\mu} & \Omega M. 
  \enddiagram\] 
The lower direction around the diagram is the definition of $\epsilon$. The upper direction 
around the diagram then shows that $\epsilon$ lifts through $\Omega i$. As $\epsilon$ is 
a homotopy equivalence, it follows that $\Omega i$ has a right homotopy inverse. 
\end{proof} 
   
\begin{example} 
\label{toricexample} 
\textit{Moment-angle manifolds I}. 
One family of examples for Theorem~\ref{inert} comes from 
toric topology. Let $K$ be a simplicial complex on $m$ vertices. For $\sigma\in K$ define 
\[(D^{2},S^{1})^{\sigma}=\prod_{i=1}^{m} Y_{i}\ \mbox{where}\ 
     Y_{i}=\left\{\begin{array}{ll} D^{2}\ \mbox{if $i\in\sigma$} \\ S^{1}\ \mbox{if $i\notin\sigma$} \end{array}\right.\]  
and define 
\[\mathcal{Z}_{K}=\bigcup_{\sigma\in K} (D^{2},S^{1})^{\sigma}.\] 
The space $\mathcal{Z}_{K}$ is called a \emph{moment-angle complex} and 
by~\cite[Proposition 4.3.5 (a)]{BP} it is always 
$2$-connected. If $K$ is the triangulation of a sphere then 
$\mathcal{Z}_{K}$ is a manifold~\cite[Theorem 4.1.4]{BP}, called a \emph{moment-angle manifold}. 
Suppose that~$K$ has a missing edge on vertices $\{i,j\}$. Let $L$ be the (full) subcomplex of $K$ 
consisting of the two vertices $\{i,j\}$. By definition,  
$\mathcal{Z}_{L}=D^{2}\times S^{1}\cup_{S^{1}\times S^{1}} S^{1}\times D^{2}\simeq S^{3}$ and, 
by~\cite{DS}, there is a retraction of $\mathcal{Z}_{L}$ 
off $\mathcal{Z}_{K}$. Thus if $K$ is the triangulation of a sphere then $\mathcal{Z}_{K}$ is 
a $2$-connected manifold and there is a map 
\(\namedright{\mathcal{Z}_{K}}{}{S^{3}}\) 
with a right homotopy inverse. Theorem~\ref{inert} therefore implies that the attaching map for 
the top cell of the manifold $\mathcal{Z}_{K}$ is inert. 
\end{example} 

\begin{example} 
\label{toricexample2} 
\textit{Moment-angle manifolds II}. 
Generalizing the previous example, suppose that $K$ is not a simplex and $\tau$ is a missing face 
of $K$ of least dimension. If $\tau=(i_{1},\ldots,i_{k})$ then $\mathcal{Z}_{K}$ is $(2k-2)$-connected. 
Let $L$ be the (full) subcomplex of $K$ on the vertices $\{i_{1},\ldots,i_{k}\}$. Then 
$L=\partial\tau$ and as in~\cite[Example 4.1.2~(4)]{BP}, 
$\mathcal{Z}_{L}\simeq S^{2k-1}$. As before, by~\cite{DS} the full subcomplex property 
implies that $\mathcal{Z}_{L}$ retracts off $\mathcal{Z}_{K}$. So if $K$ is a triangulation  
of a sphere then $\mathcal{Z}_{K}$ is a $(2k-2)$-connected manifold and there is a map 
\(\namedright{\mathcal{Z}_{K}}{}{S^{2k-1}}\) 
with a right homotopy inverse. Theorem~\ref{inert} therefore implies that the attaching map 
for the top cell of the manifold $\mathcal{Z}_{K}$ is inert. 
\end{example}

\begin{example} 
\label{Casson} 
\textit{Certain manifolds with nontrivial $\hat{A}$-genus}. 
Casson's theory of pre-fibrations~\cite{Ca} was used by Hanke, Schick and Steimle~\cite[Theorem 1.4]{HSS} 
in their study of the space of Riemannian metrics of positive sectional curvature on a closed manifold 
to construct explicit smooth fibre bundles 
\(\nameddright{F}{}{M}{h'}{S^{m}}\)  
for any $m,\ell\geq 0$, where $M$ is  a smooth closed spin manifold with non-vanishing $\hat{A}$-genus, 
$F$ is $\ell$-connected and $h'$ has a right homotopy inverse. If $\ell+1\geq m\geq 2$ then the 
long exact sequence of homotopy groups associated to the fibre bundle implies that $M$ is 
$(m-1)$-connected. Thus, as $h'$ has a right homotopy inverse, if $M$ has 
dimension~$\geq 4$ then Theorem~\ref{inert} implies that 
the attaching map for its top cell is inert. Subsequently, to prove a conjecture of Schick, a smooth 
fibre bundle 
\(\nameddright{\mathbb{H}P^{2}}{}{M}{h'}{S^{4}}\) 
was constructed by Krannich, Kupers and Randal-Williams~\cite{KKR}  
where~$M$ has dimension $12$ and $h'$ has a right homotopy inverse. Note that $M$ is $3$-connected.   
As $h'$ has a right homotopy inverse, Theorem~\ref{inert} implies that the attaching map for its top 
cell is inert. 
\end{example}

\section{A localized version of Theorem~\ref{inert} when $m$ is odd} 
\label{sec:local} 

The hypothesis in Theorem~\ref{inert} that there is a map  
\(\namedright{M}{}{S^{m}}\) 
with a right homotopy inverse may be strong. In this section and the next we show that 
in a local context this assumption is not so strong. 
As $M$ is $(m-1)$-connected, if $H_{m}(M)$ has a $\mathbb{Z}$-summand 
with generator $\overline{x}$, then the Hurewicz isomorphism implies that 
there is a map 
\(s'\colon\namedright{S^{m}}{}{M}\) 
whose Hurewicz image is $\overline{x}$. If $x\in H^{m}(M)$ is the dual of $\overline{x}$ 
then $x$ is represented by a map 
\(r\colon\namedright{M}{}{K(\mathbb{Z},m)}\). 
The composite 
\[\nameddright{S^{m}}{s'}{M}{r}{K(\mathbb{Z},m)}\]  
is homotopic to the inclusion of the bottom cell.  
   
As $m\geq 2$, the Eilenberg-Mac Lane space $K(\mathbb{Z},m)$ is not a sphere. In particular, the inclusion 
\(\namedright{S^{m}}{}{K(\mathbb{Z},m)}\)  
is not a homotopy equivalence. However, if $m$ is odd then rationally this inclusion is a 
homotopy equivalence. Aiming for something in between, let $\mathcal{P}$ be the set 
of primes $p$ such that $p\leq\frac{n-m+3}{2}$. Observe that, for a given $p$, the $p$-torsion 
homotopy group of least dimension in $\pi_{\ast}(S^{m})$ occurs in dimension $m+2p-3$. 
Therefore, if $p\notin\mathcal{P}$ then this torsion class occurs in a dimension larger than~$n$. 
Thus, if $m$ is odd, the inclusion 
\(\namedright{S^{m}}{}{K(\mathbb{Z},m)}\) 
is a homotopy equivalence in dimensions~$\leq n$ after localization away from $\mathcal{P}$. 

\begin{lemma} 
   \label{Moddlocal} 
   If $m$ is odd then, localized away from $\mathcal{P}$, there is a lift 
   \[\diagram 
         &  S^{m}\dto \\
         M\rto^-{r}\urto^{h'} & K(\mathbb{Z},m) 
      \enddiagram\] 
   for some map $h'$, and any choice of $h'$ has the property that $h'\circ s'$ is homotopic  
   to the identity map. 
\end{lemma} 

\begin{proof} 
Since the inclusion 
\(\namedright{S^{m}}{}{K(\mathbb{Z},m)}\) 
is a homotopy equivalence in dimensions~$\leq n$ after localization away from $\mathcal{P}$, 
and $M$ has dimension $n$, the existence of the a lift $h'$ for $r$ follows immediately. 
Since $r\circ s'$ is homotopic to the inclusion of the bottom cell, $h'\circ s'$ is a degree $1$ 
self-map of $S^{m}$ and so is homotopic to the identity map. 
\end{proof} 

\begin{remark} 
If $m$ is even then $\pi_{\ast}(S^{m})$ has another $\mathbb{Z}$-summand in dimension $2m-1$, 
which results in a qualitatively different case that will be partially dealt with in the next section. 
\end{remark} 

The existence of a map 
\(\namedright{M}{}{S^{m}}\) 
with a right homotopy inverse in Lemma~\ref{Moddlocal} allows for Proposition~\ref{Esplitting} 
and Theorem~\ref{inert} to be applied in a local setting. 

\begin{theorem} 
   \label{Moddinert} 
   Let $M$ be an $(m-1)$-connected, closed Poincar\'{e} Duality complex of dimension~$n$, 
   where $2\leq m<n$. If $H_{m}(M)$ has a $\mathbb{Z}$-summand and $m$ is odd then, 
   localized away from $\mathcal{P}$, there is a homotopy equivalence 
   $E\simeq (S^{n-1}\rtimes\Omega S^{m})\vee E'$ and 
   the attaching map for the top cell of $M$ is inert.~$\qqed$ 
\end{theorem} 

%

\begin{example}\textit{$2$-connected Poincar\'{e} Duality complexes}. 
Let $M$ be a $2$-connected, closed, $n$-dimensional Poincar\'{e} Duality complex. 
Suppose that $H_{3}(M)$ has a $\mathbb{Z}$-summand. The Hurewicz homomorphism implies 
that there is a map 
\(\namedright{S^{3}}{}{M}\) 
whose Hurewicz image is the given $\mathbb{Z}$-summand in $H_{3}(M)$. If this map 
has a left homotopy inverse  
\(\namedright{M}{}{S^{3}}\)  
then Theorem~\ref{inert} applies to show that the attaching map for the top cell of $M$ is inert. 
Otherwise, if $\mathcal{P}$ is the set of primes $p$ such that $p\leq\frac{n}{2}$ then, 
after localizing away from $\mathcal{P}$, Theorem~\ref{Moddinert} implies that the 
attaching map for the top cell of $M$ is inert. For example, if $M$ is a $2$-connected $8$ or $9$-manifold 
then this holds after inverting~$2$ and $3$ and if $M$ is a $2$-connected $10$-manifold then 
this holds after inverting $2$, $3$ and $5$. 
\end{example}

\section{A localized version of Theorem~\ref{inert} when $m$ is even} 
\label{sec:localeven} 

This begins as in Section~\ref{sec:local}. As $M$ is $(m-1)$-connected, 
if $H_{m}(M)$ has a $\mathbb{Z}$-summand with generator $\overline{x}$ then the 
Hurewicz isomorphism implies that there is a map 
\(s'\colon\namedright{S^{m}}{}{M}\) 
whose Hurewicz image is $\overline{x}$. If $x\in H^{m}(M)$ is the dual of $\overline{x}$ 
then $x$ is represented by a map 
\(r\colon\namedright{M}{}{K(\mathbb{Z},m)}\). 
The composite 
\[\nameddright{S^{m}}{s'}{M}{r}{K(\mathbb{Z},m)}\]  
is homotopic to the inclusion of the bottom cell. If $m$ is even, say $m=2m'$, then the inclusion 
\(\namedright{S^{m}}{}{K(\mathbb{Z},m)}\) 
is not a rational homotopy equivalence since $S^{m}$ has two nontrivial rational homotopy 
groups while $K(\mathbb{Z},m)$ only has one. In fact, there is a rational homotopy equivalence 
$K(\mathbb{Z},m)\simeq\Omega S^{2m'+1}$. So a direct analogue of the lift in Lemma~\ref{Moddlocal} 
will not hold. In Lemma~\ref{Mevenlocal} a variation will be established assuming that $x^{2}=0$. 

To do so requires some preliminary lemmas. Let 
\[H\colon\namedright{\Omega S^{2m'+1}}{}{\Omega S^{4m'+1}}\]  
be the second James-Hopf invariant. 

\begin{lemma} 
   \label{EHPQ} 
   There is a rational homotopy fibration 
   \[\nameddright{S^{2m'}}{E}{\Omega S^{2m'+1}}{H}{\Omega S^{4m'+1}}\] 
   where $E$ is the suspension map (equivalently, the inclusion of the bottom cell). 
\end{lemma}  

\begin{proof} 
Recall that  in integral homology, there are algebra isomorphisms 
$\hlgy{\Omega S^{2m'+1}}\cong\mathbb{Z}[\overline{x}]$, 
$\hlgy{\Omega S^{4m'+1}}\cong\mathbb{Z}[\overline{y}]$  
where $\vert\overline{x}\vert=2m'$ and $\vert\overline{y}\vert=4m'$, and 
$H_{\ast}(\overline{x}^{2})=\overline{y}$. Dually, in integral cohomology there are algebra isomorphisms 
$\cohlgy{\Omega S^{2m'+1}}\cong\Gamma(x)$ and $\cohlgy{\Omega S^{4m'+1}}\cong\Gamma(y)$ 
where $x$ and $y$ are the duals of $\overline{x}$ and $\overline{y}$ respectively, $\Gamma(\ \ )$ 
is the divided power algebra, and $H^{\ast}(y)=\gamma_{2}=2x^{2}$. Rationally, 
$\Gamma(x)\cong\mathbb{Q}[x]$ and the generator of $\Gamma(y)\cong\mathbb{Q}[y]$ 
may be adjusted to $y'=\frac{1}{2}y$ in order to obtain $H^{\ast}(y')=x^{2}$. As $H^{\ast}$ 
is an algebra map, the cohomology of the homotopy fibre $F$ of $H$ must be isomorphic 
to $\mathbb{Q}[x]/(x^{2})\cong\Lambda(x)$. Thus the inclusion 
\(\namedright{S^{2m'}}{}{F}\) 
of the bottom cell induces an isomorphism in rational cohomology, and hence in rational homology 
by the universal coefficient theorem, and so is a homotopy equivalence by Whitehead's Theorem. 
\end{proof} 

\begin{remark} 
\label{Hsquare} 
It will be useful to note that in the proof of Lemma~\ref{EHPQ} it was shown that the image 
of $H^{\ast}$ is twice the squaring map in degree $4m'$ integral cohomology. 
\end{remark} 

Let $u\in H^{2m'}(K(\mathbb{Z},2m'))\cong\mathbb{Z}$ be a generator. Then $u^{2}\neq 0$. Let 
\[S\colon\namedright{K(\mathbb{Z},2m')}{}{K(\mathbb{Z},4m')}\] 
represent $u^{2}$. We think of $S$ as a squaring map. There is a compatibility with the second 
James-Hopf invariant $H$. 

\begin{lemma} 
   \label{EHPsquare} 
   Localized away from $2$ there is a homotopy commutative square  
   \[\diagram 
         \Omega S^{2m'+1}\rto^-{a}\dto^{H} & K(\mathbb{Z},2m')\dto^{S} \\  
         \Omega S^{4m'+1}\rto^-{b/2} & K(\mathbb{Z},4m')  
      \enddiagram\] 
   where $a$ represents a generator of $H^{2m'}(\Omega S^{2m'+1})\cong\mathbb{Z}$ and $b$ represents 
   the $+1$ generator of $H^{4m'}(\Omega S^{4m'+1})\cong\mathbb{Z}$. 
\end{lemma} 

\begin{proof} 
As maps into Eilenberg-Mac Lane spaces represent cohomology classes, it suffices to show 
that the diagram commutes in cohomology. In integral cohomology, the definitions of $a$ and $S$ 
imply that $S\circ a$ represents~$x^{2}$, while by Remark~\ref{Hsquare}, $b\circ H$ represents $2x^{2}$. 
Thus, localized away from $2$, $S\circ a\simeq (b/2)\circ H$. 
\end{proof} 

Let $\mathcal{Q}$ be the set of primes $p$ such that $p\leq\frac{n-m+4}{2}$. Then 
if $p\notin\mathcal{Q}$, the $p$-torsion homotopy group of least dimension in $\pi_{\ast}(S^{m})$  
occurs in a dimension larger than $n+1$. Consequently, localized away from $\mathcal{Q}$, 
any $CW$-complex $X$ is homotopy equivalent to its rationalization in dimensions~$\leq n+1$. 
Note that the set $\mathcal{Q}$ differs from the set $\mathcal{P}$ in Lemma~\ref{Moddlocal} 
if and only if $n+1$ is prime. 

\begin{lemma} 
   \label{Mevenlocal}  
   Suppose that $m$ is even and 
   \(\namedright{M}{r}{K(\mathbb{Z},m)}\)  
   represents a cohomology class $x\in H^{m}(M)$ with the property that $x^{2}=0$. Then, 
   localized away from $\mathcal{Q}$, there is a lift  
   \[\diagram 
          & S^{m}\dto \\ 
          M\rto^-{r}\urto^{h'}  & K(\mathbb{Z},m) 
      \enddiagram\] 
   for some map $h'$, and any choice of $h'$ has the property that $h'\circ s'$ is homotopic 
   to the identity map. 
\end{lemma} 

\begin{proof} 
Write $m=2m'$ and consider the composite 
\(\nameddright{M}{r}{K(\mathbb{Z},2m')}{S}{K(\mathbb{Z},4m')}\). 
Since $r$ represents the cohomology class $x$, by definition of $S$, the composite 
$S\circ r$ represents $x^{2}$. By hypothesis, $x^{2}=0$, so $S\circ r$ is null homotopic. 

Since we are localized away from $\mathcal{Q}$, any space $X$ is homotopy equivalent 
to its rationalization in dimensions~$\leq n+1$. Therefore the maps 
\(\namedright{\Omega S^{2m'+1}}{a}{K(\mathbb{Z},2m')}\)  
and 
\(\namedright{\Omega S^{4m'+1}}{b}{K(\mathbb{Z},4m')}\) 
in Lemma~\ref{EHPsquare} are homotopy equivalences in dimensions~$\leq n+1$ and 
Lemma~\ref{EHPQ} implies that the homotopy fibre of $H$ is $S^{2m'}$ in dimensions~$\leq n$. 
Thus the homotopy commutative square in Lemma~\ref{EHPsquare} implies that in 
dimensions~$\leq n$ the homotopy fibre of $S$ is the same as the homotopy fibre 
of $H$, which is $S^{2m'}$. The null homotopy for $S\circ r$ therefore results in a lift 
\(h'\colon\namedright{M}{}{S^{2m'}}\) 
of $r$ through the inclusion 
\(\namedright{S^{2m'}}{}{K(\mathbb{Z},2m')}\) 
of the bottom cell. Since $r\circ s'$ is homotopic to the inclusion of the bottom cell, $h'\circ  s'$ 
is a degree $1$ self-map of $S^{2m'}$ and so is homotopic to the identity map. 
\end{proof} 

The existence of a map 
\(\namedright{M}{}{S^{m}}\) 
with a right homotopy inverse in Lemma~\ref{Mevenlocal} allows for Proposition~\ref{Esplitting} 
and Theorem~\ref{inert} to be applied localized away from $\mathcal{Q}$ and gives an 
analogue to Theorem~\ref{Moddinert}, provided $x^{2}=0$.   

\begin{theorem}
   \label{Meveninert} 
   Let $M$ be an $(m-1)$-connected, closed Poincar\'{e} Duality complex of dimension~$n$, 
   where $2\leq m<n$. Suppose that $H^{m}(M)$ has 
   a $\mathbb{Z}$-summand and $m$ is even. If the $\mathbb{Z}$-summand is generated by 
   $x\in H^{m}(M)$ and $x^{2}=0$, then localized away from $\mathcal{Q}$ 
   there is a homotopy equivalence $E\simeq (S^{n-1}\rtimes\Omega S^{m})\vee E'$ and 
   the attaching map for the top cell of $M$ is inert.~$\qqed$ 
\end{theorem} 

\begin{remark} 
\label{m=2case} 
In the case when $m=2$ there is an improvement. In the Hopf fibration 
\(\nameddright{S^{1}}{}{S^{3}}{}{S^{2}}\)  
the left map is a group homomorphism (induced by the inclusion of $\mathbb{C}$ 
into $\mathbb{H}$) and therefore classifies, resulting in a homotopy fibration 
\(\nameddright{S^{2}}{}{\mathbb{C}P^{\infty}}{}{BS^{3}}\). 
Using this fibration in place of Lemmas~\ref{EHPQ} and~\ref{EHPsquare}, 
then instead of inverting sufficiently many primes to have $\Omega S^{3}$ approximate 
$K(\mathbb{Z},2)$ we only need to invert sufficiently many primes to have $BS^{3}$  
approximate $K(\mathbb{Z},4)$. More precisely, consider the map 
\(\namedright{BS^{3}}{a}{K(\mathbb{Z},4)}\)  
representing the $+1$ generator of $H^{4}(BS^{3})\cong\mathbb{Z}$. As the least nontrivial $p$-torsion 
homotopy group of $BS^{3}$ occurs in dimension $2p+1$, the map $a$ is a $p$-local 
homotopy equivalence in dimensions~$\leq 2p$ Thus if $\mathcal{Q}'$ is the set of  
primes~$p$ such that $p<\frac{n-1}{2}$ then, after localizing away from $\mathcal{Q}'$,  
$a$ is a homotopy equivalence, implying that if $x\in H^{2}(M)$ and $x^{2}=0$ then the map 
\(\namedright{M}{r}{\mathbb{C}P^{\infty}}\) 
representing $x$ lifts to $S^{2}$. Using this in place of Lemma~\ref{EHPQ} allows for 
the statement of Theorem~\ref{Meveninert} to hold with respect to $\mathcal{Q}'$. 
\end{remark} 

\begin{example}\textit{Certain simply-connected $6$ and $8$-manifolds}. 
Let $M$ be a simply-connected closed $n$-dimensional Poincar\'{e} Duality 
complex. Suppose that $H^{2}(M)$ has a $\mathbb{Z}$-summand generated by a class $x$ 
satisfying $x^{2}=0$. Dualizing, $H_{2}(M)$ has a $\mathbb{Z}$-summand generated by a dual 
class~$\overline{x}$ and the Hurewicz homomorphism implies that there is a map 
\(\namedright{S^{2}}{}{M}\) 
whose Hurewicz image is~$\overline{x}$. If this map has a left homotopy inverse 
\(\namedright{M}{}{S^{2}}\)  
then Theorem~\ref{inert} applies to show that the attaching map for the top cell of $M$ is inert. 
Otherwise, if $\mathcal{Q}'$ is the set of primes $p$ such that $p<\frac{n-1}{2}$ then, after localizing 
away from $\mathcal{P}$, the modification of Theorem~\ref{Meveninert} in Remark~\ref{m=2case}  
implies that the attaching map for the top cell of~$M$ is inert. For example, if $M$ is a 
simply-connected $6$-manifold with a generator $x\in H^{2}(M)$ satisfying $x^{2}=0$ then 
after inverting $2$ the attaching map for the top cell of $M$ is inert; if $M$ is a simply-connected 
$8$-manifold with a generator $x\in H^{2}(M)$ satisfying $x^{2}=0$ then after inverting $2$ and $3$ 
the attaching map for the top cell of $M$ is inert. 
\end{example}

\section{Inert attaching maps II}  
\label{sec:torus} 

In this section we prove an inertness result for simply-connected Poincar\'{e} Duality complexes 
having $\mathbb{Z}$-summands in degree~$2$ cohomology. This depends on the fact 
that $S^{1}$ has a classifying space. The main statement is Theorem~\ref{circleinert}, which will 
then be iterated to give Theorem~\ref{torusinert}. First, some background information is needed. 
Let $BS^{1}$ be the classifying space of $S^{1}$. 

\begin{lemma} 
   \label{NPD} 
   Let $M$ be a simply-connected closed Poincar\'{e} Duality complex of dimension~$n$. 
   Suppose that there is a homotopy fibration 
   \[\nameddright{N}{}{M}{h}{BS^{1}}\] 
   where $\Omega h$ has a right homotopy inverse. Then $N$ is a simply-connected closed 
   Poincar\'{e} Duality complex of dimension $n+1$. 
\end{lemma} 

\begin{proof} 
It was claimed by Quinn~\cite{Q}, with an explicit argument given by Gottlieb~\cite{Go}, that if 
\(\nameddright{F}{}{E}{}{B}\) 
is a homotopy fibration of finite $CW$-complexes and both $B$ and $F$ are 
Poincar\'{e} Duality complexes then so is $E$. In our case, extending the homotopy fibration 
\(\nameddright{N}{}{M}{h}{BS^{1}}\) 
one step to the left gives a principle fibration 
\(\nameddright{S^{1}}{}{N}{}{M}\). 
As both $M$ and $S^{1}$ are Poincar\'{e} Duality complexes, the Quinn-Gottlieb result 
implies that $N$ is too. 

Consider the cohomology Serre spectral sequence applied to the homotopy fibration 
\(\nameddright{S^{1}}{}{N}{}{M}\). 
If $z\in H^{n}(M)\cong\mathbb{Z}$ and $\iota\in H^{1}(S^{1})\cong\mathbb{Z}$ are generators 
then $z\otimes\iota$ is on the $E_{2}$-page of the spectral sequence and survives the spectral 
sequence since it is neither the source nor target of a nontrivial differential for degree reasons. 
This results in a $\mathbb{Z}$ class in $H^{n+1}(N)$ of highest dimension, implying that $N$ 
is $(n+1)$-dimensional.  

The fact that $\Omega h$ has a right homotopy inverse implies that the map 
\(\namedright{S^{1}}{}{N}\) 
is null homotopic, implying in turn that the map 
\(\namedright{N}{}{M}\) 
induces an injection on homotopy groups. In particular, as $M$ is simply-connected, so is $N$. 
\end{proof} 

Now suppose that $M$ is a simply-connected closed Poincar\'{e} Duality complex 
of dimension~$n$ and there is a homotopy fibration 
\(\nameddright{N}{}{M}{h}{BS^{1}}\) 
where $\Omega h$ has a right homotopy inverse. By Lemma~\ref{NPD}, $N$ is a simply-connected 
closed Poincar\'{e} Duality complex of dimension~$n+1$. Recall that $\overline{M}$ 
is the $(n-1)$-skeleton of $M$ and 
\(\namedright{\overline{M}}{i}{M}\) 
is the skeletal inclusion. Define the space $Q$ and the maps $a$, $b$ and $c$ by the homotopy 
fibration diagram  
\begin{equation} 
  \label{Qdef} 
  \diagram 
     S^{1}\rto\ddouble & Q\rto^-{a}\dto^{b} & \overline{M}\rto^-{h\circ i}\dto^{i} & BS^{1}\ddouble \\ 
     S^{1}\rto & N\rto^-{c} & M\rto^-{h} &  BS^{1}.  
  \enddiagram 
\end{equation}  
The space $Q$ will be a useful intermediate space so some of its properties are recorded. 

\begin{lemma} 
   \label{MMbarcompatsplit} 
   The map 
   \(\llnamedright{\Omega\overline{M}}{\Omega h\circ\Omega i}{S^{1}}\) 
   has a right homotopy inverse and there is a homotopy commutative diagram 
   \[\diagram 
          S^{1}\times\Omega Q\rto^-{\overline{e}}\dto^{1\times\Omega b} 
                & \Omega\overline{M}\dto^{\Omega i} \\ 
          S^{1}\times\Omega N\rto^-{e} & \Omega M 
     \enddiagram\] 
   where $e$ and $\overline{e}$ are homotopy equivalences. 
\end{lemma} 

\begin{proof} 
By hypothesis, $\Omega h$ has a right homotopy inverse 
\[s\colon\namedright{S^{1}}{}{\Omega M}.\] 
For skeletal reasons, $s$ lifts through $\Omega i$ to a map 
\[\overline{s}\colon\namedright{S^{1}}{}{\Omega\overline{M}}.\]   
Then $\overline{s}$ is a right homotopy inverse for $\Omega h\circ\Omega i$ since 
$(\Omega h\circ\Omega i)\circ\overline{s}\simeq\Omega h\circ(\Omega i\circ\overline{s})=
     \Omega h\circ s$. 
     
Next, consider the diagram 
\begin{equation} 
  \label{compatS1decomp} 
  \diagram 
    S^{1}\times\Omega Q\rto^-{\overline{s}\times\Omega a}\dto^{1\times\Omega b} 
          & \Omega\overline{M}\times\Omega\overline{M}\rto^-{\mu}\dto^{\Omega i\times\Omega i} 
          & \Omega\overline{M}\dto^{\Omega i} \\ 
    S^{1}\times\Omega N\rto^-{s\times\Omega c} & \Omega M\times\Omega M\rto^-{\mu} & \Omega M  
  \enddiagram 
\end{equation} 
where $\mu$ is the loop multiplication. The left square homotopy commutes 
since $\Omega i\circ\overline{s}\simeq s$ and by the middle square in~(\ref{Qdef}). The 
right square homotopy commutes since $\Omega i$ is a loop map. The top row is a homotopy 
equivalence since it realizes a splitting of the homotopy fibration 
\(\llnameddright{\Omega Q}{\Omega a}{\Omega\overline{M}}{\Omega h\circ\Omega i}{S^{1}}\) 
and the bottom row is a homotopy equivalence since it realizes a splitting of the homotopy fibration 
\(\nameddright{\Omega N}{\Omega c}{\Omega M}{\Omega h}{S^{1}}\). 
\end{proof} 

Now consider the diagram 
\begin{equation} 
  \label{Qdata} 
  \diagram 
     &  Q\rto^-{b}\dto^{a} &  N\dto^{c} \\ 
     S^{n-1}\rto^-{f} & \overline{M}\rto^-{i}\dto^{h\circ i} & M\dto^{h} \\ 
     & BS^{1}\rdouble & BS^{1}. 
  \enddiagram 
\end{equation} 
The middle row is the homotopy cofibration that attaches the top cell to $M$ and the 
columns form a homotopy fibration diagram by~(\ref{Qdef}). Thus this is a diagram of data 
in the sense of~(\ref{data}). By Lemma~\ref{MMbarcompatsplit}, $\Omega h\circ\Omega i$ has 
a right homotopy inverse. Therefore Theorem~\ref{BT} implies that there is a homotopy cofibration 
\begin{equation} 
  \label{QNcofib} 
  \nameddright{S^{n-1}\rtimes S^{1}}{\theta}{Q}{b}{N} 
\end{equation}  
for a map $\theta$ with the property that the restriction of $\theta$ to $S^{n-1}$ is a lift of 
\(\namedright{S^{n-1}}{f}{\overline{M}}\) 
through $a$. 

We will modify the homotopy cofibration~(\ref{QNcofib}) to get better control over the space $Q$. 
By Lemma~\ref{NPD}, $N$ is an $(n+1)$-dimensional simply-connected closed Poincar\'{e} 
Duality complex. Therefore there is a homotopy cofibration 
\[\nameddright{S^{n}}{f'}{\overline{N}}{i'}{N}\] 
where $\overline{N}$ is the $n$-skeleton of $N$, $i'$ is the skeletal inclusion and $f'$ is the 
attaching map for the top cell. Define $j$ by the composite 
\[j\colon\nameddright{S^{n-1}}{i_{1}}{S^{n-1}\rtimes S^{1}}{\theta}{Q}\] 
where $i_{1}$ is the inclusion. 

\begin{lemma} 
   \label{QNbarcofib} 
   There is a homotopy cofibration 
   \(\nameddright{S^{n-1}}{j}{Q}{\overline{b}}{\overline{N}}\) 
   for a map $\overline{b}$ with the property that the composite 
   \(\nameddright{Q}{\overline{b}}{\overline{N}}{i'}{N}\) 
   is homotopic to $b$. 
\end{lemma} 

\begin{proof} 
Consider the homotopy cofibration  
\(\nameddright{S^{n-1}\rtimes S^{1}}{\theta}{Q}{b}{N}\). 
Regard $S^{n-1}\rtimes S^{1}$ as $S^{n-1}\vee S^{n}$. Attaching one cell at a time gives homotopy cofibrations 
\[\nameddright{S^{n-1}}{j}{Q}{}{Q'}\]  
\[\nameddright{S^{n}}{}{Q'}{}{N},\] 
where the first of these defines the space $Q'$ and the composite 
\(\nameddright{Q}{}{Q'}{}{N}\) 
is homotopic to $b$. Observe that the homotopy fibration 
\(\nameddright{S^{1}}{}{Q}{a}{\overline{M}}\) 
implies that $Q$ is at most $(n-1)$-dimensional since $\overline{M}$ is $(n-2)$-dimensional. Thus 
the homotopy cofibration defining $Q'$ implies that it is at most $n$-dimensional. Therefore, as $N$ 
is $(n+1)$-dimensional, the homotopy cofibration 
\(\nameddright{S^{n}}{}{Q'}{}{N}\) 
is attaching the $(n+1)$-cell to $N$. This implies that $Q'$ must account for all the cells of $N$ 
in dimensions~$\leq n$, that is, $Q'\simeq\overline{N}$. Thus, substituting this into the statements above,  
there is a homotopy cofibration 
\(\nameddright{S^{n-1}}{j}{Q}{\overline{b}}{\overline{N}}\),  
where $\overline{b}$ is the map to the cofibre, and the composite 
\(\nameddright{Q}{\overline{b}}{\overline{N}}{i'}{N}\) 
is homotopic to $b$. 
\end{proof} 

The key property of $Q$ is the following splitting, whose proof is lengthy. 

\begin{proposition} 
   \label{QNbarsplit} 
   The homotopy cofibration in Lemma~\ref{QNbarcofib} splits to give a homotopy equivalence 
   $Q\simeq S^{n-1}\vee\overline{N}$. 
\end{proposition} 

\begin{proof} 
We proceed in steps. 
\medskip 

\noindent 
\textit{Step 1: Setting up}. 
Consider the map 
\(\namedright{M}{h}{BS^{1}}\). 
Since $BS^{1}\simeq K(\mathbb{Z},2)$, $h$ represents a class $x\in H^{2}(M)$. Since $\Omega h$ 
has a right homotopy inverse, $x$ must be the generator of a $\mathbb{Z}$-summand. 
By Poincar\'{e} Duality, there is a class $y\in H^{n-2}(M)$ such that $x\cup y=z$, where $z\in H^{n}(M)$ 
is the fundamental class. Note that $y$ generates a $\mathbb{Z}$-summand in $H^{n-2}(M)$. 
For dimensional reasons, since~$\overline{M}$ is the $(n-1)$-skeleton of $M$, both $x$ and $y$ 
correspond to classes in $H^{\ast}(\overline{M})$; the same names will be used for the corresponding 
generators. Since $M$ is simply-connected, $H_{1}(M)\cong 0$ so Poincar\'{e} Duality implies 
that $H^{n-1}(M)\cong 0$. Therefore $\overline{M}$ is actually $(n-2)$-dimensional. The attaching 
map for the $(n-2)$-cells in $\overline{M}$ has a connecting map 
\(\namedright{\overline{M}}{\delta'}{\bigvee_{i=1}^{d} S^{n-2}}\) 
that collapses out the $(n-3)$-skeleton, and this connecting map has a homotopy co-action 
\(\namedright{\overline{M}}{\psi'}{\overline{M}\vee(\bigvee_{i=1}^{d} S^{n-2})}\) 
with respect to $\delta'$. Observe that $y$ is in the image of $(\delta')^{\ast}$. Changing 
$\bigvee_{i=1}^{d} S^{n-2}$ by a self-equivalence if necessary, we may assume that the 
restriction of $(\delta')^{\ast}$ to the $i=1$ summand has image $y$. Let $\delta$ be the composite 
\[\delta\colon\nameddright{\overline{M}}{\delta'}{\bigvee_{i=1}^{d} S^{n-2}}{p_{1}}{S^{n-2}}\] 
where $p_{1}$ is the pinch map to the first wedge summand, and let $\psi$ be the composite 
\[\psi\colon\lnameddright{\overline{M}}{\psi'}{\overline{M}\vee(\bigvee_{i=1}^{d} S^{n-2})}{1\vee p_{1}} 
     {\overline{M}\vee S^{n-2}}.\] 
Then $\delta^{\ast}$ has image $y$ and $\psi$ is a homotopy co-action with respect to $\delta$. 
\medskip 

\noindent 
\textit{Step 2: A homotopy co-action for $Q$}. 
Apply Proposition~\ref{co-action} with the homotopy fibration 
\(\nameddright{E}{}{X}{h}{Z}\)  
and the homotopy co-action 
\(\namedright{X}{\psi}{X\vee Y}\)  
being 
\(\nameddright{Q}{}{\overline{M}}{h}{BS^{1}}\) 
and 
\(\namedright{\overline{M}}{\psi}{\overline{M}\vee S^{n-2}}\) 
respectively. Then there are homotopy fibration diagrams 
\[\diagram 
    Q\rto^-{\psi_{Q}}\dto^{a} & Q\vee(S^{n-2}\rtimes S^{1})\dto & & 
         Q\rto^-{\delta_{Q}}\dto^{a} & S^{n-2}\rtimes S^{1}\dto  \\ 
    \overline{M}\rto^-{\psi}\dto^{h} & \overline{M}\vee S^{n-2}\dto^{\overline{h}} 
         &  & \overline{M}\rto^-{(h\vee 1)\circ\psi}\dto^{h} & BS^{1}\vee S^{n-2}\dto \\ 
    BS^{1}\rdouble & BS^{1} & & BS^{1}\rdouble & BS^{1} 
  \enddiagram\] 
where $\psi_{Q}$ is a homotopy co-action with respect to $\delta_{Q}$. 
\medskip 

\noindent 
\textit{Step 3: A left homotopy inverse for 
\(\namedright{S^{n-1}}{j}{Q}\)}. 
Consider the diagram 
\begin{equation} 
  \label{Qid} 
  \diagram 
     S^{n-1}\rto^-{i_{1}}\drrto_{f} & S^{n-1}\rtimes S^{1}\rto^-{\theta} & Q\rto^-{\delta_{Q}}\dto^{a} 
          & S^{n-2}\rtimes S^{1}\dto \\ 
     & & \overline{M}\rto^-{(h\vee 1)\circ\psi} & BS^{1}\vee  S^{n-2}. 
  \enddiagram 
\end{equation}  
Since the restriction of $\theta$ is a lift of $f$ through $a$, the left triangle homotopy commutes. 
The right square homotopy commutes since it is the top square of the right diagram in Step~$2$. 
By definition, $j=\theta\circ i_{1}$.   We will show that the composite 
\(\nameddright{Q}{\delta_{Q}}{S^{n-2}\rtimes S^{1}}{q}{S^{n-1}}\), 
where $q$ is the quotient map to $S^{n-2}\wedge S^{1}\simeq S^{n-1}$,  
is a left homotopy inverse for $j$.  

By~\cite{Ga}, there is a homotopy fibration 
\[\nameddright{\Sigma\Omega BS^{1}\wedge\Omega S^{n-2}}{W}{BS^{1}\vee S^{n-2}}{}{BS^{1}\times S^{n-2}}\] 
where the right map is the inclusion of the wedge into the product and $W$ is the Whitehead product 
of the maps 
\(\nameddright{\Sigma\Omega BS^{1}}{ev}{BS^{1}}{i_{1}}{BS^{1}\vee  S^{n-2}}\) 
and 
\(\nameddright{\Sigma\Omega S^{n-2}}{ev}{S^{n-2}}{i_{2}}{BS^{1}\vee S^{n-2}}\), 
with $ev$ being the evaluation map and $i_{1}$ and $i_{2}$ being the inclusions of the left and right 
wedge summands respectively.  Further, this homotopy fibration splits after looping to give a homotopy equivalence 
\[\Omega(BS^{1}\vee  S^{n-2})\simeq 
     S^{1}\times\Omega S^{n-2}\times\Omega(\Sigma S^{1}\wedge\Omega S^{n-2}),\] 
where $\Omega BS^{1}$ has been simplified to $S^{1}$. Note that 
$\Sigma S^{1}\wedge\Omega S^{n-2}\simeq S^{n-1}\vee R$ 
where $R$ is $(n-1)$-connected and the restriction of $W$ to $S^{n-1}$ is the composite  
\(w\colon\namedright{S^{n-1}}{}{S^{2}\vee S^{n-2}}\hookrightarrow BS^{1}\vee S^{n-2}\), 
where the left map is the Whitehead product whose cofibre is $S^{2}\times S^{n-2}$. 

The decomposition of $\Omega(BS^{1}\vee S^{n-2})$ may be used to identify a map 
\(\namedright{S^{m}}{}{BS^{1}\vee S^{n-2}}\) 
by taking its adjoint, identifying the maps to the factors in $\Omega(BS^{1}\vee S^{n-2})$, 
and then taking adjoints to get back to the original map. Doing this for the composite 
$(h\vee 1)\circ\psi\circ f$ in~(\ref{Qid}) gives $(h\vee 1)\circ\psi\circ f\simeq s\cdot\widetilde{\eta} + t\cdot w$  
for some $s\in\mathbb{Z}/2\mathbb{Z}$ and $t\in\mathbb{Z}$, where $\widetilde{\eta}$ is the composite 
\(\nameddright{S^{n-1}}{\eta}{S^{n-2}}{i_{2}}{BS^{1}\vee S^{n-2}}\),  
with $\eta$ representing a generator of $\pi_{n-1}(S^{n-2})\cong\mathbb{Z}/2\mathbb{Z}$. 

We claim that $t=1$. Since the Whitehead product 
\(\namedright{S^{n-1}}{}{S^{2}\vee  S^{n-2}}\) 
is detected by the cup product in cohomology, so is $w$. Consider the homotopy cofibration diagram 
\[\diagram 
     S^{n-1}\rrto^-{f}\ddouble & & \overline{M}\rto^-{i}\dto^{(h\vee 1)\circ\psi} & M\dto^{\varphi} \\ 
     S^{n-1}\rrto^-{(h\vee 1)\circ\psi\circ f} & & BS^{1}\vee S^{n-2}\rto & C 
  \enddiagram\] 
that defines the space $C$ and the map $\varphi$. In cohomology, the definitions of $h$ and $\psi$ 
imply that $((h\vee 1)\circ\psi)^{\ast}$ sends the degree $2$ generator $x'\in H^{2}(BS^{1})$ to 
$x\in H^{2}(\overline{M})$ and sends the degree $2n-2$ generator $y'\in H^{2n-2}(S^{2n-2})$ to  
$y\in H^{n-2}(\overline{M})$. Since $x\cup y=z$ in $H^{\ast}(M)$, we obtain $x'\cup y'=z'$ 
in $H^{\ast}(C)$ with $\varphi^{\ast}(z')=z$. As $w$ detects the cup product, we must have $t=1$, 
implying that $(h\vee 1)\circ\psi\circ f\simeq s\cdot\widetilde{\eta} + w$. 

On the other hand, the homotopy commutativity of~(\ref{Qid}) implies that $(h\vee 1)\circ\psi\circ f$ 
is homotopic to the composite 
\(\namedddright{S^{n-1}}{i_{1}}{S^{n-1}\rtimes S^{1}}{\theta}{Q}{\delta_{Q}}{S^{n-2}\rtimes S^{1}} 
    \stackrel{\gamma}{\longrightarrow} BS^{1}\vee  S^{n-2}\), 
where $\gamma$ is simply a name for the map appearing in~(\ref{Qid}). Regard 
$S^{n-2}\rtimes S^{1}$ as $S^{n-2}\vee S^{n-1}$. Then the retriction of~$\gamma$ to $S^{n-2}$ 
is the inclusion $i_{2}$ of the second wedge summand and the restriction to $S^{n-1}$ is $w$. 
The composite $\delta_{Q}\circ\theta\circ i_{1}$ equals $s'\cdot\eta+t'\cdot\iota$, where 
$s'\in\mathbb{Z}/2\mathbb{Z}$, $\iota$ is the identity map on $S^{n-1}$, and $t'\in\mathbb{Z}$. 
Thus $\gamma\circ\delta_{Q}\circ\theta\circ i_{1}\simeq s'\cdot\widetilde{\eta}+t'\cdot w$. 
But $\gamma\circ\delta_{Q}\circ\theta\circ i_{1}\simeq (h\vee 1)\circ\psi\circ f\simeq s\cdot\widetilde{\eta}+w$, 
so it must be the case that $t'=1$. Thus, $\delta_{Q}\circ\theta\circ i_{1}\simeq s'\cdot\widetilde{\eta} +\iota$. 
Composing with the quotient map 
\(\namedright{S^{n-2}\rtimes S^{1}}{q}{S^{n-1}}\) 
then gives $q\circ\delta_{Q}\circ\theta\circ i_{1}\simeq\iota$. Hence $q\circ\delta_{Q}$ is a 
left homotopy inverse for $\theta\circ i_{1}=j$. 
\medskip 

\noindent 
\textit{Step 4: The splitting of $Q$}. By Lemma~\ref{QNbarcofib} there is a homotopy cofibration 
\(\nameddright{S^{n-1}}{j}{Q}{\overline{b}}{\overline{N}}\). 
By Step~$3$, a left homotopy inverse for $j$ is 
\(\llnamedright{Q}{q\circ\delta_{Q}}{S^{n-1}}\). 
By Step~$2$, $Q$ has a homotopy co-action 
\(\namedright{Q}{\psi_{Q}}{Q\vee (S^{2n-2}\rtimes S^{1})}\) 
with respect to $\delta_{Q}$, implying that $p_{2}\circ\psi_{Q}\simeq\delta_{Q}$. Therefore 
the composite 
\[\Psi\colon\nameddright{Q}{\psi_{Q}}{Q\vee (S^{2n-1}\rtimes S^{1})}{\overline{b}\vee q}{\overline{N}\vee S^{n-1}}\]   
has the property that $p_{1}\circ\Psi\simeq\overline{b}$ and $p_{2}\circ\Psi\simeq q\circ\delta_{Q}$ (the 
left homotopy inverse for $j$). Thus $\Psi$ splits the homotopy cofibration 
\(\nameddright{S^{n-1}}{j}{Q}{\overline{b}}{\overline{N}}\), 
implying that it induces an isomorphism in homology and so is a homotopy equivalence by 
Whitehead's Theorem. 
\end{proof} 

\begin{theorem}[Theorem~\ref{introcircleinert} in the Introduction]  
   \label{circleinert} 
   Let $M$ be a simply-connected closed Poincar\'{e} Duality complex of dimension $n$. 
   Suppose that there is a homotopy fibration 
   \[\nameddright{N}{}{M}{h}{BS^{1}}\] 
   where $\Omega h$ has a right homotopy inverse. If the attaching map for the top cell of $N$ 
   is inert then the attaching map for the top cell of $M$ is inert. 
\end{theorem} 

\begin{proof} 
By hypothesis, the attaching map for the top cell of $N$ is inert, which by definition of inertness, 
means that the skeletal inclusion  
\(\namedright{\overline{N}}{i'}{N}\)  
has the property that $\Omega i'$ has a right homotopy inverse 
\(\mathfrak{s}\colon\namedright{\Omega N}{}{\Omega\overline{N}}\). 
By definition, 
\(\namedright{Q}{\overline{b}}{\overline{N}}\) 
is a lift of 
\(\namedright{Q}{b}{N}\)  
through $i'$. By Proposition~\ref{QNbarsplit}, $\overline{b}$ has a right homotopy inverse 
\(t\colon\namedright{\overline{N}}{}{Q}\). 
Therefore the composite 
\[\namedddright{\Omega N}{\mathfrak{s}}{\Omega\overline{N}}{\Omega t}{\Omega Q}{\Omega b}{\Omega N}\] 
satisfies 
$\Omega b\circ\Omega t\circ\mathfrak{s}\simeq\Omega i\circ\Omega\overline{b}\circ\Omega t\circ\mathfrak{s}\simeq 
     \Omega i\circ\mathfrak{s}\simeq 1$, 
where $1$ refers to the identity map on $\Omega N$. Now consider the diagram 
\[\diagram 
      S^{1}\times\Omega M\rto^-{e^{-1}} & S^{1}\times\Omega N\rto^-{1\times(\Omega t\circ\mathfrak{s})}\drdouble 
           & S^{1}\times\Omega Q\rto^-{\overline{e}}\dto^{1\times\Omega b} & \Omega\overline{M}\dto^{\Omega i} \\ 
      & & S^{1}\times\Omega N\rto^-{e} & \Omega M  
  \enddiagram\] 
where $\overline{e}$ and $e$ are the homotopy equivalences in~(\ref{compatS1decomp}). 
We have just seen that the left triangle homotopy commutes. The right square homotopy commutes 
by~(\ref{compatS1decomp}). The homotopy commutativity of the diagram therefore implies that 
$\overline{e}\circ (1\times(\Omega t\circ\mathfrak{s}))\circ e^{-1}$ is a right homotopy inverse for $\Omega i$. 
Thus, by definition of inertness, the attaching map for the top cell of $M$ is inert. 
\end{proof} 

Theorem~\ref{circleinert} can be iterated and combined with Theorem~\ref{inert}. Let $T^{\ell}$ be the torus 
with $\ell$ circle factors and let $BT^{\ell}$ be its classifying space. 

\begin{theorem} 
   \label{torusinert} 
   Let $M$ be a simply-connected closed Poincar\'{e} Duality complex of dimension $n$. 
   Suppose that there is a homotopy fibration 
   \[\nameddright{N}{}{M}{h}{BT^{\ell}}\] 
   where $\Omega h$ has a right homotopy inverse, $N$ is an $(m-1)$-connected, closed 
   Poincar\'{e} Duality complex for $m\geq 3$ odd, and there is a map 
   \(\namedright{N}{}{S^{m}}\) 
   that has  a right homotopy inverse. Then the attaching map for the top cell of $M$ is inert. 
\end{theorem} 

\begin{proof} 
Since $\Omega h$ has a right homotopy inverse and $N$ is at least $2$-connected, $h^{\ast}$ induces 
an isomorphism on degree $2$ cohomology. Therefore $H^{2}(M)\cong\bigoplus_{i=1}^{\ell}\mathbb{Z}$. 
As this is a free group, the universal coefficient theorem implies that $H_{2}(M)\cong H^{2}(M)$. 
Since $M$ is simply-connected, the Hurewicz isomorphism 
implies that $\pi_{2}(M)\cong H_{2}(M)$. Therefore there is a map 
\(j\colon\namedright{\bigvee_{i=1}^{\ell} S^{2}}{}{M}\) 
such that the composite 
\(\nameddright{\bigvee_{i=1}^{\ell} S^{2}}{j}{M}{h}{BT^{\ell}}\) 
induces an isomorphism on $H^{2}$. Changing $\bigvee_{i=1}^{\ell} S^{2}$ by a self-equivalence 
if necessary, we may assume that $h\circ j$ is the composite of inclusions 
\(\bigvee_{i=1}^{\ell} S^{2}\hookrightarrow\prod_{i=1}^{\ell} S^{2}\hookrightarrow 
     \prod_{i=1}^{\ell} BS^{1}\simeq BT^{\ell}\).  
     
Let $M_{1}=M$,  $N_{1}=N$, $j_{1}=j$ and $h_{1}=h$, so there is a homotopy fibration 
\(\nameddright{N_{1}}{}{M_{1}}{h_{1}}{\prod_{i=1}^{\ell} BS^{1}}\) 
and $h_{1}\circ j_{1}$ is the inclusion 
\({\bigvee_{i=1}^{\ell} S^{2}}\hookrightarrow BT^{\ell}\). 
Let $\mathfrak{h}_{1}$ be the composite 
\(\nameddright{M_{1}}{h_{1}}{BT^{\ell}}{\pi_{1}}{BS^{1}}\) 
where $\pi_{1}$ is the projection to the first factor. Define the space $M_{2}$ by the homotopy fibration 
\[\nameddright{M_{2}}{}{M_{1}}{\mathfrak{h}_{1}}{BS^{1}}.\] 
Since $\Omega h$ has a right homotopy inverse, so does $\Omega\mathfrak{h}_{1}$. 
Therefore Theorem~\ref{circleinert} states that if the attaching map for the top cell of $M_{2}$ 
is inert then so is that for $M_{1}$. We now show that $M_{2}$ is of a form similar to $M_{1}$. Let $h_{2}$ 
be the composite 
\[h_{2}\colon\namedddright{M_{2}}{}{M_{1}}{h_{1}}{\prod_{i=1}^{\ell} BS^{1}}{\pi}{\prod_{i=2}^{\ell} BS^{1}},\]  
where $\pi$ is the projection. Since $h_{1}\circ j_{1}$ is a composite of inclusions, by definition of $M_{2}$ 
as the homotopy fibre of~$\mathfrak{h}_{1}$ there is a lift 
\[\diagram 
      \bigvee_{i=2}^{\ell} S^{2}\rto\dto^{j_{2}} & \bigvee_{i=1}^{\ell} S^{2}\dto^{j_{1}}  \\ 
      M_{2}\rto & M_{1} 
  \enddiagram\] 
for a map $j_{2}$ that induces an isomorphism on $\pi_{2}$. Notice that $h_{2}\circ j_{2}$ is the composite 
of inclusions 
\(\bigvee_{i=2}^{\ell} S^{2}\hookrightarrow\prod_{i=2}^{\ell} S^{2}\hookrightarrow 
     \prod_{i=2}^{\ell} BS^{1}\simeq BT^{\ell}\). 
Note also that this description of $h_{2}\circ j_{2}$ also implies that~$\Omega h_{2}$ has a 
right homotopy inverse via the composite 
\(\namedright{\prod_{i=2}^{\ell} S^{1}}{s_{1}}{\Omega(\bigvee_{i=2}^{\ell} S^{2})} 
     \stackrel{\Omega j_{2}\circ\Omega h_{2}}{\lllarrow}\prod_{i=2}^{\ell} S^{1}\),  
 where $s_{1}$ exists by the Hilton-Milnor Theorem and the fact that 
 $\Omega S^{2}\simeq S^{1}\times\Omega S^{3}$. Thus we may form a homotopy fibration 
 \[\nameddright{N_{2}}{}{M_{2}}{h_{2}}{\prod_{i=2}^{\ell} BS^{1}}\] 
 that defines $N_{2}$, with $\Omega h_{2}$ having a right homotopy inverse. Now the argument may be iterated. 
 
 After $\ell$ iterations, the homotopy fibre $N_{\ell}$ is homotopic to the given space $N$ that 
 is the homotopy fibre of 
 \(\namedright{M}{h}{BT^{\ell}}\). 
 Let $M_{\ell+1}=N$. We have obtained a chain of Poincar\'{e} Duality complexes  
 $M_{\ell+1},M_{\ell},\ldots,M_{1}$ 
 with the property that, for each $2\leq k\leq\ell+1$, if the attaching map for the top cell of $M_{k}$ is 
 inert then the attaching map for the top cell of $M_{k-1}$ is inert. Thus, collectively, if the attaching 
 map for the top cell of $N=M_{\ell+1}$ is inert then the attaching map for the top cell of $M=M_{1}$ is inert. By 
 hypothesis, $N$ is $(m-1)$-connected for $m\geq 3$ odd and there is a map 
 \(\namedright{N}{}{S^{m}}\) 
 that has a right homotopy inverse. Thus the attaching map for the top cell of $N$ is inert 
 by Theorem~\ref{inert}, and hence the attaching map for the top cell of $M$ is inert. 
\end{proof} 

An interesting application of Theorem~\ref{torusinert} is to quasi-toric manifolds. 

\begin{example} 
\label{quasitoric} 
A $2n$-dimensional manifold has a \emph{locally standard} $T^{n}$-action if locally it is 
the standard action of $T^{n}$ on $\mathbb{C}^{n}$. An $n$-dimensional convex polytope $P$ 
is \emph{simple} if it has exactly $n$ facets intersecting at each vertex. A \emph{quasi-toric manifold} 
over $P$ is a closed, smooth $2n$-dimensional manifold $M^{2n}$ that has a smooth locally standard 
$T^{n}$-action for which the orbit space $M^{2n}/T^{n}$ is homeomorphic to $P$. These were 
first constructed in~\cite{DJ} and they exist in abundance. Let $P$ be a simple polytope 
of dimension $n$ and having $m$ facets. Let $K$ be the simplicial complex that is Alexander 
dual to the boundary of $P$. By~\cite[Proposition 7.3.12]{BP}, a quasi-toric manifold $M^{2n}$ arises 
as a quotient $M^{2n}\cong\mathcal{Z}_{K}/T^{m-n}$ for some subtorus $T^{m-n}\subseteq T^{m}$  
that acts freely on the corresponding moment-angle complex $\mathbb{Z}_{K}$. This implies that 
there is a principle $T^{m-n}$-fibration 
\[\nameddright{T^{m-n}}{}{\mathcal{Z}_{K}}{}{M^{2n}}.\]  
As this is principle, it is classified by a map 
\(\namedright{M^{2n}}{}{BT^{m-n}}\)  
and there is a homotopy fibration 
\begin{equation} 
  \label{ZMfib} 
  \nameddright{\mathcal{Z}_{K}}{}{M^{2n}}{h}{BT^{m-n}}.  
\end{equation} 
As in Examples~\ref{toricexample} and~\ref{toricexample2}, $\mathcal{Z}_{K}$ is $(m-1)$-connected 
for $m\geq 3$ odd and if $K$ is not a simplex then there is a map 
\(\namedright{\mathcal{Z}_{K}}{}{S^{m}}\)  
with a right homotopy inverse. The connectivity of $\mathcal{Z}_{K}$ implies that~$h$ induces 
an isomorphism on $\pi_{2}$ and therefore $\Omega h$ induces an isomorphism on $\pi_{1}$. 
Using the loop space structure to multiply Hurewicz images then gives a map 
\(\namedright{\prod_{i=1}^{m-n} S^{1}}{}{\Omega M^{2n}}\) 
that is a right homotopy inverse for $\Omega h$. Thus Theorem~\ref{torusinert} applied to~(\ref{ZMfib}) 
implies that the attaching map for the top cell of $M^{2n}$ is inert. 
\end{example}

\section{The rational case} 
\label{sec:rational} 

In this section we show that our methods recover Halperin and Lemaire's result that any simply-connected 
closed Poincar\'{e} Duality complex $M$ has the attaching map for its top cell being inert provided 
the rational cohomology of $M$ is not generated by a single element. This involves a generalization 
of Theorem~\ref{circleinert}. 

Throughout this section all spaces will be rationalized. In particular, for any $m\geq 1$ this implies 
that $S^{2m-1}$ is an Eilenberg-Mac Lane space $K(\mathbb{Q},2m-1)$. Write $BS^{2m-1}$ for 
its classifying space. 

 \begin{theorem} 
   \label{EMinert} 
   Let $M$ be a $(2m-1)$-connected, closed Poincar\'{e} Duality complex of dimension $n$, 
   where $m\geq 1$ and $n\geq 4$. Suppose that $M$ is not a sphere and there is a homotopy fibration 
   \[\nameddright{N}{}{M}{h}{BS^{2m-1}}\] 
   where $\Omega h$ has a right homotopy inverse. If the attaching map for the top cell of $N$ 
   is inert then the attaching map for the top cell of $M$ is inert. 
\end{theorem} 

\begin{proof} 
The proof is similar to that for Theorem~\ref{circleinert}. Note that Lemmas~\ref{NPD}, 
\ref{MMbarcompatsplit}, and~\ref{QNbarcofib} all hold in the rational case with $BS^{1}$ simply  
replaced by $BS^{2m-1}$. Steps 1, 2 and 4 in the proof of Proposition~\ref{QNbarsplit} also 
hold with the same replacement. For Step 3 of Proposition~\ref{QNbarsplit}, 
Diagram~(\ref{Qid}) becomes 
\[\diagram 
     S^{n-1}\rto^-{i_{1}}\drrto_{f} & S^{n-1}\rtimes S^{2m-1}\rto^-{\theta} & Q\rto^-{\delta_{Q}}\dto^{a} 
          & S^{n-2m}\rtimes S^{2m-1}\dto \\ 
     & & \overline{M}\rto^-{(h\vee 1)\circ\psi} & BS^{2m-1}\vee  S^{n-2m}. 
  \enddiagram\] 
As $M$ is not a sphere, Poincar\'{e} Duality implies that the bottom cell of $M$ occurs in a degree 
no more than half the dimension of $M$. Therefore $4m\leq n$. 

There are two cases. First, if $4m=n$ then $M$ is a $(2m-1)$-connnected $4m$-dimensional 
Poincar\'{e} Duality complex and $n-2m=2m$. In this case the Hilton-Milnor Theorem implies that 
the composite $(h\vee 1)\circ\psi\circ f$ is homotopic to $s\cdot\omega+t\cdot w$ where $\omega$ 
generates $\pi_{4m-1}(S^{2m})\cong\mathbb{Q}$, $w$ is the composite 
\(\namedright{S^{4m-1}}{}{S^{2m}\vee S^{2m}}\hookrightarrow BS^{2m-1}\vee S^{2m}\) 
in which the left map is the Whitehead product whose cofibre is $S^{2m}\times S^{2m}$, 
and $s,t\in\mathbb{Q}$. The key to the remainder of the argument in the proof of Step 3 
in Proposition~\ref{QNbarsplit} is that $t=1$, and the same argument works in this case as well. 
Second, if $4m<n$ then the cell of least dimension in $M$ is in a degree strictly less than half the 
dimension of $M$. Thus $2m<n-2m$, in which case the composite $(h\vee 1)\circ\psi\circ f$ 
is homotopic to $t\cdot w$, and the argument as before in Step 3 of Proposition~\ref{QNbarsplit} 
again implies that $t=1$. In both cases we obtain $t=1$ which then allows for the argument 
that $q\circ\delta_{Q}$ is a left homotopy inverse for $\theta\circ i_{i}=j$. 

The proof now proceeds exactly as in that for Theorem~\ref{circleinert} with $BS^{1}$ 
replaced by $BS^{2m-1}$. 
\end{proof} 

\begin{theorem} 
   \label{rationalcase} 
   Let $M$ be a simply-connected closed Poincar\'{e} Duality complex of dimension $n$, 
   where $n\geq 4$. If the rational cohomology of $M$ is not generated by a single element then 
   the attaching map for the top cell of $M$ is inert. 
\end{theorem} 

\begin{proof} 
Suppose that $M$ is $(m-1)$-connected. As $H^{\ast}(M;\mathbb{Q})$ is not generated by a single 
element, $M$ is not a sphere, so Poincar\'{e} Duality implies that $2m\leq n$. Let $x\in H^{m}(M;\mathbb{Q})$ 
be a generator. Then~$x$ is represented by a map 
\(r\colon\namedright{M}{}{K(\mathbb{Q},m)}\). 
Let $\bar{x}\in H^{m}(M;\mathbb{Q})$ be the dual of $x$. The Hurewicz isomorphism implies that 
there is a map 
\(s\colon\namedright{S^{m}}{}{M}\) 
with Hurewicz image $\bar{x}$. The duality between~$\bar{x}$ and $x$ implies that the composite  
\(\nameddright{S^{m}}{s}{M}{r}{K(\mathbb{Q},m)}\) 
is homotopic to the inclusion of the bottom cell. 
\medskip  

\noindent 
\textit{Case 1: $m$ is odd}. Then $K(\mathbb{Q},m)\simeq S^{m}$, so we may regard $r$ as a map 
\(\namedright{M}{r}{S^{m}}\).  
As $r\circ s$ has degree one, $r\circ s$ is homotopic to the identity map. Thus $r$ has a right 
homotopy inverse. The rational version of either Theorem~\ref{inert} or Theorem~\ref{Moddinert} 
then implies that the attaching map for the top cell of $M$ is inert. 
\medskip 

\noindent 
\textit{Case 2: $m$ is even}. Then we may write $K(\mathbb{Q},2m)$ as $BS^{2m-1}$ and 
regard $r$ as a map 
\(\namedright{M}{r}{BS^{2m-1}}\). 
After looping the composite 
\(\llnameddright{S^{2m-1}}{E}{\Omega S^{2m}}{\Omega (r\circ s)}{S^{2m-1}}\) 
is a homotopy equivalence, where~$E$ is the suspension. Thus $\Omega r$ has a right 
homotopy inverse. If $M_{1}$ is the homotopy fibre of $r$, then by~\cite{Go,Q} $M_{1}$ 
is a closed Poincar\'{e} Duality complex, so Theorem~\ref{EMinert} 
implies that if the attaching map for the top cell of $M_{1}$ is inert then the attaching 
map for the top cell of $M$ is inert. Therefore we are reduced to considering $M_{1}$. The 
argument may now be iterated for~$M_{1}$, giving a sequence of Poincar\'{e} Duality complexes 
$M_{k}\longrightarrow M_{k-1}\longrightarrow\cdots\longrightarrow M_{1}\longrightarrow M_{0}=M$, 
where, for $1\leq i\leq k$, $M_{i}$ is the homotopy fibre of a map 
\(\namedright{M_{i-1}}{r_{i-1}}{BS^{2m_{i-1}-1}}\)  
with $\Omega r_{i-1}$ having a right homotopy inverse, and the attaching map for the top cell 
of $M_{i-1}$ being inert provided the attaching map for the top cell of~$M_{i}$ is inert. The 
iteration stops if $M_{k}$ is $(m'-1)$-connected for $m'$ odd or if $M_{k}$ is 
a sphere. For if $M_{k}$ is $(m'-1)$-connected for $m'$ odd then $M_{k}$ satisfies Case 1, in 
which case the attaching map for the top cell of $M_{k}$ is inert and hence the attaching 
map for the top cell of $M$ is inert. Otherwise, $M_{k}$ is a sphere. 

To go further, return to $M_{1}$ and consider rational cohomology. The map $r$ represents 
a class $x\in H^{2m}(M;\mathbb{Q})$. Fibering $r$, there is an associated principal fibration 
\(\nameddright{S^{2m-1}}{}{M_{1}}{}{M}\). 
Choose a generator $\iota\in H^{2m-1}(S^{2m-1};\mathbb{Q})$ so that $\iota$ transgresses to $x$ 
in the cohomology Serre spectral sequence. Then $d^{2m}(\iota)=x$. Suppose that $x^{\ell}\neq 0$ 
but $x^{\ell+1}=0$ for some $\ell\geq  1$. As the differential is a derivation we obtain 
$d^{2m}(x^{\ell-1}\otimes\iota)=x^{\ell}\neq 0$ and $d^{2m}(x^{\ell}\otimes\iota)=x^{\ell+1}=0$. 
Thus $x^{\ell}\otimes\iota$ survives to the $E_{2m+1}$-page, and for degree reasons all other 
differentials are zero so $x^{\ell}\otimes\iota$ survives to the $E_{\infty}$-page. Thus $x^{\ell}\otimes\iota$ 
represents an odd degree element $y_{1}\in H^{\ast}(M_{1};\mathbb{Q})$. At the next step in the iteration, 
the map 
\(\namedright{M_{1}}{r_{1}}{BS^{2m_{1}-1}}\) 
represents a class $x_{1}\in H^{2m_{1}}(M_{1};\mathbb{Q})$, and in the cohomology Serre 
spectral sequence for the associated principal fibration 
\(\nameddright{S^{2m_{1}-1}}{}{M_{2}}{}{M_{1}}\), 
a generator $\iota_{1}\in H^{2m_{1}-1}(S^{2m_{1}-1};\mathbb{Q})$ may be chosen to transgress 
to $x_{1}$. The image of $d^{2m_{1}}$ lies in even degrees in $H^{\ast}(M_{1};\mathbb{Q})$, 
so as $y_{1}$ is of odd degree it cannot be in the image of $d^{2m_{1}}$. All other differentials 
are zero for degree reasons, so $y_{1}$ must survive the spectral sequence, and therefore it 
represents an element in $H^{\ast}(M_{2};\mathbb{Q})$. The same argument applied iteratively 
implies that $y_{1}$ represents an element in $H^{\ast}(M_{k};\mathbb{Q})$. Further, the same 
argument applied to $x_{1}\in H^{2m_{1}}(M_{1};\mathbb{Q})$ implies that there is an odd degree 
element $y_{2}\in H^{\ast}(M_{2};\mathbb{Q})$ that also represents an element in 
$H^{\ast}(M_{k};\mathbb{Q})$. Thus if $k\geq 2$ then $H^{\ast}(M_{k};\mathbb{Q})$ has two 
distinct odd degree elements so $M_{k}$ cannot be a sphere, in which case we must be in Case 1, 
implying - as above - that the attaching map for the top cell of $M$ is inert. 
The one remaining option is when $k=1$, which implies that $M_{1}$ is a sphere, and as 
$y\in H^{\ast}(M_{1};\mathbb{Q})$ is of degree $2m(\ell+1)-1$, we must have 
$M_{1}\simeq S^{2m(\ell+1)-1}$. But then the spectral sequence argument that produced $y_{1}$ 
now implies that $H^{\ast}(M;\mathbb{Q})\cong\mathbb{Q}[x]/(x^{\ell+1})$. By hypothesis, the 
rational cohomology of~$M$ is not generated by a single element, so this case cannot occur. 
\end{proof}

\section{Two applications} 
\label{sec:apps} 

We consider two applications of the inertness property. The first is to identify the homotopy fibre 
of the inclusion 
\(\namedright{\overline{M}}{i}{M}\), 
and the second is to generate more examples via connected sums and determine the 
homotopy type of their loop spaces. 

\subsection{The homotopy fibre of the inclusion \(\namedright{\overline{M}}{i}{M}\)} 
The inertness property of the attaching map $f$ of the top cell for $M$ says that the map 
\(\namedright{\Omega\overline{M}}{\Omega i}{\Omega M}\)  
has a right homotopy inverse. The complementary factor can be identified; more 
strongly, the homotopy fibre of $i$ can be identified. 

Given maps 
\(f\colon\namedright{\Sigma A}{}{X}\) 
and 
\(g\colon\namedright{\Sigma B}{}{X}\) 
their \emph{Whitehead product} is denoted by  
\[[f,g]\colon\namedright{\Sigma A\wedge B}{}{X}.\] 
Let 
\(ev\colon\namedright{\Sigma\Omega X}{}{X}\)  
be the canonical evaluation map. The following was proved in~\cite{BT2}. 

\begin{lemma} 
   \label{inertfibre} 
   Suppose there is a homotopy cofibration 
   \(\nameddright{\Sigma A}{f}{X}{i}{Y}\)  
   and $\Omega i$ has a right homotopy inverse 
   \(s\colon\namedright{\Omega Y}{}{\Omega X}\). 
   Then there is a  homotopy fibration 
   \[\llnameddright{\Sigma A\vee(\Sigma A\wedge\Omega Y)}{f\perp [f,\gamma]}{X}{i}{Y}\] 
   where $\gamma$ is the composite 
   \(\nameddright{\Sigma\Omega Y}{\Sigma s}{\Sigma\Omega X}{ev}{X}\). 
   Further, this fibration splits after looping to give a homotopy equivalence 
   \[\Omega X\simeq\Omega Y\times\Omega(\Sigma A\vee(\Sigma A\wedge\Omega Y)).\] 
\end{lemma} 
\vspace{-1cm}~$\qqed$\bigskip 

In our case, for the Poincar\'{e} Duality complexes in Theorem~\ref{inert}, there is 
a homotopy cofibration 
\(\nameddright{S^{n-1}}{f}{\overline{M}}{}{M}\) 
that attaches the top cell and $\Omega i$ has a right homotopy inverse. Lemma~\ref{inertfibre} 
then immediately implies the following. 

\begin{proposition} 
   \label{Mfibre} 
   Let $M$ be a simply-connnected closed Poincar\'{e} Duality complex satisfying 
   the hypotheses of Theorem~\ref{inert}. Then there is a homotopy fibration 
   \[\llnameddright{S^{n-1}\vee(S^{n-1}\wedge\Omega M)}{f\perp[f,\gamma]}{\overline{M}}{i}{M}\] 
   that splits after looping to give a homotopy equivalence 
   $\Omega\overline{M}\simeq\Omega M\times\Omega(S^{n-1}\vee\Sigma^{n-1}\Omega M)$. 
   
   If instead~$M$ satisfies the hypotheses of Theorem~\ref{Moddinert} then this homotopy  
   fibration and homotopy equivalence for $\Omega\overline{M}$ hold after localization 
   away from $\mathcal{P}$.~$\qqed$ 
\end{proposition} 

For example, rationally, any suspension is homotopy equivalent to a wedge of spheres. 
Thus the homotopy fibre of $i$ in Proposition~\ref{Mfibre} is rationally a wedge of spheres. 
More generally, if $\Omega M$ was known to be homotopy equivalent to a product of 
spheres and loops on spheres, either in the integral or local case, then the homotopy fibre 
of $i$ in Proposition~\ref{Mfibre} would correspondingly be homotopy equivalent to a wedge 
of spheres. An explicit example is the following. 

\begin{example}  
\label{quasitoric2} 
\textit{An elaboration on certain quasi-toric manifolds}. 
By Example~\ref{quasitoric}, the attaching map for the top cell of a quasi-toric manifold $M^{2n}$ 
is inert. As noted in that example, there is a homotopy fibration 
\(\nameddright{\mathcal{Z}_{K}}{}{M^{2n}}{h}{BT^{m-n}}\) 
where $\Omega h$ has a right homotopy inverse and $\mathcal{Z}_{K}$ is a moment-angle manifold. 
Suppose that the associated polytope $P=M^{2n}/T^{n}$ is a product of simplices, 
$P=\prod_{j=1}^{\ell}\Delta^{m_{\ell}}$. Then, by 
combining~\cite[Example 2.2.9 (4), Proposition 4.1.3 and Example 4.1.2 (4)]{BP}, there are homeomorphisms  
\[\mathcal{Z}_{K}\cong\prod_{j=1}^{\ell}\mathcal{Z}_{\partial\Delta^{m_{j}}}\cong\prod_{j=1}^{\ell} S^{2m_{j}+1}.\] 
Using the right homotopy inverse for $\Omega h$ we therefore obtain homotopy equivalences 
\[\Omega M^{2n}\simeq\big(\prod_{i=1}^{m-n} S^{1}\big)\times\Omega\mathcal{Z}_{K}\simeq 
      \big(\prod_{i=1}^{m-n} S^{1}\big)\times\big(\prod_{j=1}^{\ell}\Omega S^{2m_{j}+1}\big).\] 
Consequently, $\Sigma^{2n-1}\Omega M^{2n}$ is homotopy equivalent to a wedge $W$ of spheres. 
Therefore, by Proposition~\ref{Mfibre} there is a homotopy fibration 
\[\nameddright{S^{2n-1}\vee W}{}{\overline{M^{2n}}}{i}{M^{2n}}\] 
that splits after looping to give a homotopy equivalence 
$\Omega\overline{M^{2n}}\simeq\Omega M^{2n}\times\Omega(S^{2n-1}\vee W)$. Note that 
all this happens without any localization. 
\end{example} 
\medskip

\subsection{Connected sums} 
Theorems~\ref{inert},~\ref{Moddinert} and~\ref{Meveninert} give many examples 
of Poincar\'{e} Duality complexes with the property that the attaching map for the top cell is inert. 
These can be used to generate more examples due to the following result~\cite[Theorem 9.1]{T}. Note 
that if $M$ and $N$ are simply-connected closed $n$-dimensional Poincar\'{e} 
Duality complexes with $(n-1)$-skeletons $\overline{M}$ and $\overline{N}$ respectively, 
then the $(n-1)$-skeleton of $M\conn N$ is $\overline{M}\vee\overline{N}$. 

\begin{theorem} 
   \label{inertconnsum} 
   Let $M$ and $N$ be simply-connected closed $n$-dimensional Poincar\'{e} Duality complexes, 
   where $n\geq 2$. Let $\overline{M}$ and $\overline{N}$ be the $(n-1)$-skeletons of $M$ 
   and $N$ respectively. If the attaching map~$f$ for the top cell of~$M$ is inert then the following hold: 
   \begin{letterlist} 
      \item there is a homotopy equivalence 
              $\Omega(M\conn N)\simeq\Omega M\times\Omega(\Omega M\ltimes\overline{N})$; 
      \item the attaching map for the top cell of $M\conn N$ is inert: the loop map 
               \(\namedright{\Omega(\overline{M}\vee\overline{N})}{}{\Omega(M\conn N)}\) 
               has a right homotopy inverse. 
    \end{letterlist} 
\end{theorem} 

Theorem~\ref{inertconnsum}~(b) says that if the attaching map for the top cell of $M$ is inert 
then so is the attaching map for the top cell of $M\conn N$ for any $N$, regardless of the behaviour 
of the attaching map for the top cell of $N$. This lets us generate new examples freely. For instance, 
if $M^{2n}$ is a quasi-toric manifold of dimension $2n$ then, by Example~\ref{quasitoric}, the attaching 
map for its top cell is inert. So if~$N$ is any $2n$-dimensional, simply-connected closed Poincar\'{e} 
Duality complex then the attaching map for the top cell of $M^{2n}\conn N$ is inert. 

Further, Theorem~\ref{inertconnsum}~(a) gives a homotopy decomposition for $\Omega(M\conn N)$ 
in terms of $\Omega M$ and~$\overline{N}$. For instance, continuing Example~\ref{quasitoric2}, 
for any quasi-toric manifold $M^{2n}$ associated to a simply polytope that is a product of simplices and 
for any $N$ as in the previous paragraph, there is a homotopy equivalence 
$\Omega(M^{2n}\conn N)\simeq\Omega M^{2n}\times\Omega(\Omega M^{2n}\ltimes\overline{N})$, 
which may be refined by substituting in the homotopy equivalence 
$\Omega M^{2n}\simeq\prod_{i=1}^{m-n} S^{1}\times\prod_{j=1}^{\ell}\Omega S^{2m_{j}+1}$  
from Example~\ref{quasitoric2}. 

Theorem~\ref{inertconnsum} also works equally well in a local setting, in which case the 
connected sum is given a topological interpretation as the sum of the attaching maps for 
the top cells.

\bibliographystyle{amsalpha}

\end{document}